\documentclass[12pt]{article}

\usepackage{graphicx}
\usepackage{appendix}

\usepackage{amssymb}	
\usepackage{amsmath}    
\allowdisplaybreaks

\newcommand*{\AN}[1]{(#1)}
\newcommand*{\BL}[1]{[#1]}

\newenvironment{proof}[1][Proof:]{\begin{trivlist} 
 		\item[\hskip \labelsep {\bfseries #1}]}{\end{trivlist}} 
\newcommand{\qed}{\nobreak \ifvmode \relax \else \ifdim\lastskip<1.5em \hskip-\lastskip \hskip1.5em plus0em minus0.5em \fi \nobreak \vrule height0.75em width0.5em depth0.25em\fi} 
\def\0{\bf \0}
\def\A{{\bf A}}

\def\D{{\bf D}}

\def\H{{\bf H}}
\def\I{{\bf I}}

\def\0{{\bf 0}}

\def\R{\mathbb{R}}
\def\S{{\bf S}}
\def\T{{\bf T}}

\def\X{{\bf X}}

\def\Z{{\bf Z}}
\def\a{{\bf a}}
\def\b{{\bf b}}
\def\c{{\bf c}}
\def\d{{\bf d}}
\def\e{{\bf e}}

\def\g{{\bf g}}

\def\p{{\bf p}}
\def\q{{\bf q}}
\def\r{{\bf r}}
\def\s{{\bf s}}

\def\v{{\bf v}}
\def\w{{\bf w}}
\def\x{{\bf x}}
\def\y{{\bf y}}
\def\z{{\bf z}}

\def\Tr{{\rm T}}
\def\T{{\rm T}}
\def\diag{{\rm diag}}

\newtheorem{algorithm}{Algorithm}[section]

\newtheorem{theorem}{Theorem}[section]
\newtheorem{lemma}{Lemma}[section]
\newtheorem{proposition}{Proposition}[section]

\newtheorem{remark}{Remark}[section]

\usepackage{color}
\usepackage[normalem]{ulem}

\begin{document}

\title{A polynomial time infeasible interior-point arc-search algorithm for convex optimization}

\author{Yaguang Yang}


\maketitle

\begin{abstract}
This paper proposes an infeasible interior-point algorithm for
the convex optimization problem using arc-search techniques. 
The proposed algorithm simultaneously selects the centering 
parameter and the step size, aiming at optimizing the 
performance in every iteration. Analytic formulas for the 
arc-search are provided to make the arc-search method 
very efficient. The convergence of the algorithm is proved 
and a polynomial bound of the algorithm is established.
The preliminary numerical test results indicate that the algorithm
is efficient and effective.

\vspace{0.2in}
\noindent
{\bf keywords:} Infeasible interior-point algorithm; Arc-search; Convex optimization

\vspace{0.1in}
\noindent
{MSC class:  90C51 90C25}
\end{abstract}

\section{ Introduction}
Because of the great success of the interior-point methods for linear 
programming (LP) problems \cite{wright97}, the methods
have been extended to more optimization problems, such
as linear complementarity problem \cite{kmy89}, convex
quadratic optimization problem \cite{ma89}, 
semidefinite programming problem \cite{aliz91}, convex 
nonlinear optimization problem \cite{ye87}, and non-convex
nonlinear programming problem \cite{ettz96} and many
references therein.

There are two types of the interior-point methods based on the 
property of the starting points of the algorithm. The
``feasible'' interior-point method starts with a feasible initial point
and is much easier to analyze the convergence properties but 
it needs an expensive ``phase-I'' process to find a feasible 
staring point. The ``infeasible'' interior-point method
does not need a feasible initial point which is computationally
attractive but its convergence analysis is much more difficult 
and it needs more demanding assumptions in the convergence
analysis \cite{wright97}. For decades, people have realized
\cite{lms91,lms92,yang17} that infeasible interior-point method
is a better strategy than feasible interior-point method for 
LPs if an initial point is not available. 

Another proven strategy of the interior-point methods is to use 
the central path to guide a series iterates to an optimal solution.
Computing the central path of an optimization problem, however,
is very expensive. Most path-following type interior-point 
algorithms use line segment to approximate the central path and
search the optimizer along this line segment. Clearly, this is not a
good strategy because the central path is a curve. Therefore, 
this author proposed an arc-search technique for interior-point 
method for LPs \cite{yang09}. The main idea in the arc-search 
technique is to efficiently and robustly approximate the central 
path using an arc of part of an ellipse and to search the 
next iterate along the arc. Since the central path is geometrically
a high dimensional curve, the arc can fit the central path better than a line.

Recently, researchers have applied arc-search techniques to 
different optimization problems. For example, an arc-search 
interior-point algorithm proposed in \cite{yy18} shows that 
it has better polynomial bound and its numerical test is more
attractive than a line-search type interior-point algorithm for LPs.
In \cite{yang18}, this author showed that an interior-point algorithm
using the arc-search technique achieves the best polynomial bound
for all interior-point methods, feasible or infeasible, and is 
numerically competitive to the well-known Mehrotra’s algorithm.
Researchers have applied the arc-search technique to the
linear complementarity problem~\cite{kheirfam17}, convex 
quadratic programming~\cite{yang11,zhl21}, symmetric 
programming~\cite{ylz17}, semidefinite programming
\cite{zyzlh19,kheirfam21}, and nonlinear programming 
problem \cite{yiy21}. All these results showed that the 
arc-search method performs better than the counterpart,
the line search method.

In this paper, we extend the arc-search techniques to the 
convex nonlinear optimization problem for which various 
line-search interior point algorithms have been developed in 
\cite{aabbdk19,agj00,Bubeck15,cb21,fu08,hertog12,jarre92,kz93,Monteiro94},
because many application problems can be formulated as
a convex nonlinear optimization problem \cite{bv14,llp17,taylor15}.
Although a polynomial bound has been proved for a feasible interior-point
algorithm for the convex nonlinear optimization problem \cite{kz93}, 
to our best knowledge, there is no polynomial bound for infeasible
interior-point algorithms for the convex nonlinear optimization problem 
because the latter is much more difficult \cite{wright97}.
We propose an arc-search infeasible interior-point algorithm for 
the convex nonlinear optimization problem and discuss the 
convergence property. We show that this algorithm
converges under mild conditions in a polynomial bound of
$O(n^{1.5} \log(1/\epsilon))$. 

The remainder of the paper is organized as follows.
Section~\ref{sec:description} introduces the problem to be discussed.
Section~\ref{sec:algorithm} describes the proposed arc-search
algorithm. Section~\ref{sec:convergence} discusses
its convergence properties. Section \ref{sec:test} contains the
materials about the Matlab implementation and preliminary
numerical test results. Section~\ref{sec:conclusion} 
summarizes the conclusion of the paper. A method about
optimal selection of the centering parameter and the step 
size at the same time is provided in Appendix \ref{sec:selection}. 

\section{Problem description}\label{sec:description}

In the remainder of the paper, we use a superscript $\T$ for the 
transpose of a vector or a matrix, and we use a tuple to denote 
a stacked vectors, for example,
$(\x,\y)$ stands for $[\x^{\T},\y^{\T}]^{\T}$.
For a vector $\x \in \R^n$, we denote by $\X \in \R^{n \times n}$  
a diagonal matrix whose diagonal elements are $\x$,
and by $\min(\x)$ and $\max(\x)$ the minimum and 
maximum values of $\x$ respectively. For two vectors 
$\x \in \R^n$ and $\y \in \R^n$, we use $\x \circ \y \in \R^n$
to denote the element-wise product of $\x$ and $\y$.
Let $\R_+^n$ ($\R_{++}^n$)  denote the space of nonnegative vectors 
(positive vectors, respectively), and $\e$  denote a vector of all ones with 
appropriate dimension. We will use superscript $k$ for the vector 
iteration count and subscript $k$ for the scalar iteration count, 
for example, $\x^k$ is the value of the vector variable $\x$ at
iteration $k$, and $\mu_k$ is the value of the scalar variable $\mu$
at iteration $k$.

We consider the following convex programming problem with linear constraints:
\begin{align}
\begin{array}{rcl}
\min &:& f(\x) \\
\textrm{s.t.} &:& \A_E \x = \b_E, \\
& &  \A_I \x \ge \b_I,
\end{array}
\label{CP}
\end{align}
where $f: \R^n \rightarrow \R$ is a nonlinear convex function of 
$\x \in \R^n$, which is differentiable up to the third order; 
$\A_E \in \R^{m \times  n}$, $\A_I \in \R^{p \times n}$, 
$m<n$, $\b_E  \in \R^{m}$, and $\b_I  \in \R^{p}$ are given 
constant matrices and vectors; and the decision variable vector
is $\x$. We assume that the row of $\A_E$ is full rank, which is 
standard because we can remove dependent rows in finite
operations bounded by a polynomial of $m$ and $n$. 

Following the treatment of \cite{yiy21}, we convert the inequality 
constraints $\A_I \x \ge \b_I$ into equality constraints by introducing 
a slack vector $\s \ge \0$ as follows:
\begin{align}
\begin{array}{rcl}
\min &:&  f(\x)  \\
\textrm{s.t.} &:& \A_E \x = \b_E, \\
& &  \A_I \x - \s = \b_I, \hspace{0.1in} \s \ge \0.
\end{array}
\label{CP1}
\end{align}
Let Lagrangian multipliers of system (\ref{CP1}) be denoted by
$\y \in \R^m, \w \in \R_+^p$ and $\z \in \R_+^p$, and let the tuple 
$\v  = (\x,\y,\w,\s,\z) \in \R^{n+m+3p} $ to represent the decision variables 
and multipliers. Then, the Lagrangian function of (\ref{CP1}) is given by
\begin{equation*}
L(\v)=f(\x)+\y^{\T}(\A_E \x - \b_E)-\w^{\T}( \A_I \x - \s - \b_I)-\z^{\T}\s.
\label{lagrangian1}
\end{equation*}
Hence, we have the gradients of Lagrangian with respect to $\x$ and $\s$
given as follows:
\begin{equation}
\nabla_{\x} L(\v)=\nabla f(\x)+\A_E^{\T} \y-\A_I^{\T}  \w,
\hspace{0.1in} \nabla_{\s} L(\v)=\w-\z.
\label{dlagrangian}
\end{equation}
Let $\mu$ be the duality measure defined as
\begin{equation}
\mu = \frac{\s^{\T} \z}{p}.
\label{mu}
\end{equation}
The KKT conditions of (\ref{CP1}) are
\begin{align}
\g(\v)  =  \0,  \
(\w, \s, \z) \in \R_+^{3p}, 
\label{KKT1}
\end{align}
where $\g:\R^{n+m+3p} \to \R^{n+m+3p}$ is defined by 
\begin{equation}
\g(\v) = \left[ \begin{array}{c}
\nabla_{\x} L(\v)  \\
\A_E \x - \b_E  \\
\A_I \x - \s - \b_I \\
\w-\z  \\
\Z \s
\end{array} \right]
\approx  \left[ \begin{array}{c}
\H \x +\A_E^{\T} \y -\A_I^{\T} \w   \\
\A_E \x - \b_E  \\
\A_I \x - \s - \b_I \\
\w-\z  \\
\Z \s
\end{array} \right]
= 
\left[ \begin{array}{c}
\r_C  \\
\r_E  \\
\r_I \\
\w-\z  \\
p\mu \e
\end{array} \right],
\label{defineG}
\end{equation}
and 
\begin{subequations}
\begin{gather}
\r_C^k=\H \x^k+\A_E^{\T} \y^k-\A_I^{\T}\w^k,\\
\r_E^k=\A_E \x^k -\b_E, \\
\r_I^k=\A_I \x^k - \s^k - \b_I
\end{gather}
\label{residuals}
\end{subequations}
are the approximated residual of the gradient of the Lagrangian function
at $\v^k$ as defined in (\ref{dlagrangian}), the residual 
of the equality constraints, and the residual of the inequality constraints, 
respectively. The last row of (\ref{defineG}) requires that the
iterate follows the central path as closer as possible.

\begin{remark}
Please note that $\r_C^k$ is not defined as a strict (but an 
approximate) residual of the gradient of the Lagrangian function 
at $\v^k$. This modification is for the purpose to obtain a 
convergent algorithm. 
\end{remark}

In view of (\ref{dlagrangian}), we have 
$\nabla_{\x}^2 L(\v)=\nabla_{\x}^2 f(\x):=\H_{\x}$ which is
a positive definite matrix depending on $\x$ because
$f(\x)$ is a nonlinear convex function. To simplify the notation, 
we will write $\H$ for $\H_{\x}$ but $\H$ needs to be updated
in every iteration. The Jacobian of $\g$ is given by
\begin{align*}
\g' (\v) = \left[ \begin{array}{ccccc}
\nabla_{\x}^2 f(\x) & \A_E^{\T} & -\A_I^{\T} & \0 & \0 \\
\A_E &\0 & \0 & \0 & \0  \\
\A_I &\0 & \0 & -\I & \0 \\
\0  & \0  &  \I  &  \0  &  -\I  \\
\0 & \0 & \0  &  \Z & \S 
\end{array} \right]
= \left[ \begin{array}{ccccc}
\H & \A_E^{\T} & -\A_I^{\T} & \0 & \0 \\
\A_E &\0 & \0 & \0 & \0  \\
\A_I &\0 & \0 & -\I & \0 \\
\0  & \0  &  \I  &  \0  &  -\I  \\
\0 & \0 & \0  &  \Z & \S 
\end{array} \right]. 
\label{firstJacobian}
\end{align*}
Let $\a_i$ be the $i$th row of $\A_I$, $b_i$ be the 
$i$th element of $\b_I$, $i \in \{1,\ldots,p\}$, 
$\a_j$ be the $j$th row of $\A_E$,  $b_j$ be the 
$j$th element of $\b_E$, $j \in \{1,\ldots,m\}$.
Let 
\begin{equation*}
I(\x)=\left\{i \in \{1,\ldots,p\} : \a_i \x =b_i \right\}
\label{index}
\end{equation*}
be the index set of active inequality constraints at $\x \in \R^n$.
It is easy to check that the following properties hold for
problem (\ref{CP}).

\begin{proposition}
Assume that (a) $\A_E$ is full rank, (b) the constraints set of system (\ref{CP})
is not empty, (c) $f(\x)$ is differentiable up to the third order and is 
locally Lipschitz continuous at optimal solution $\bar{\x}$, then system 
(\ref{CP}) has the following properties.
\begin{itemize}
\item[(P1)] There exists $\bar{\v}=(\bar{\x},\bar{\y},\bar{\w},\bar{\s},\bar{\z})$, 
an optimal solution and its associate multipliers of (\ref{CP1}), i.e., 
KKT conditions (\ref{KKT1}) has a solution.
\item[(P2)] $\g(\x)$ are differentiable up to the second order. 
In addition, $\g(\x)$ is locally Lipschitz continuous at $\bar{\x}$.
\end{itemize}
\label{prop1}
\hfill \qed
\end{proposition}

For the convergence analysis, we make the following
assumptions for Problem (\ref{CP}).

\vspace{0.05in}
\noindent
{\bf Assumptions:}
\begin{itemize}
\item[(A1)]  The set
$\{ \a_j : j =1, \ldots, m\} \cup \{ \a_i  : i \in I(\bar{\x}) \}$
is linearly independent.
\item[(A2)] For all $\boldsymbol\zeta  \in \R^n \backslash \{0\}$,
we have $\boldsymbol\zeta^{\T} \nabla_{\x}^2 L(\bar{\v})  \boldsymbol\zeta >0$.
\item[(A3)]  For each
$i \in \{ 1, \ldots, p \}$, we have $\bar{z}_i+\bar{s}_i >0$
and $\bar{z}_i \bar{s}_i =0$.
\end{itemize}
Here (A1) is the linear independence
constraint qualification (LICQ); (A2) is the second order sufficient
conditions, which is true because Problem (\ref{CP}) is a convex
optimization problem; and (A3) is strict complementarity. All these properties
are standard and used in convergence analysis in \cite{ettz96,nw06}.
These properties assure the nonsingularity 
of the Jacobian matrix at the optimal solution $\bar{\v}$.

\begin{theorem}\label{nonsingular}
If Conditions (P1), (A1), (A2), and (A3) hold, then the Jacobian 
matrix $\g'(\bar{\v})$ is nonsingular.
\end{theorem}
\begin{proof} 
Let $(\hat{\a},\hat{\b},\hat{\c},\hat{\d},\hat{\e})$ be a constant vector that satisfies
\begin{equation}
\left[ \begin{array}{c}
\nabla_{\x}^2 L(\bar{\v}) \\
\A_E \\
\A_I  \\
\0 \\ 
\0 
\end{array} \right] \hat{\a} +
\left[ \begin{array}{c}
\A_E^{\T} \\
\0 \\
\0 \\
\0 \\ 
\0
\end{array} \right] \hat{\b} +
\left[ \begin{array}{c}
-\A_I^{\T} \\
\0 \\
\0 \\
\I \\ 
\0 
\end{array} \right] \hat{\c} +
\left[ \begin{array}{c}
\0 \\
\0 \\
-\I  \\
\0 \\ 
\bar{\Z}  
\end{array} \right] \hat{\d} +
\left[ \begin{array}{c}
\0 \\
\0  \\
\0 \\
-\I  \\
\bar{\S}
\end{array} \right] \hat{\e} =\0.
\label{independent}
\end{equation}
To show the nonsigularity of $\g'(\bar{\v})$,
it is enough to show that  (\ref{independent}) 
holds only if $(\hat{\a},\hat{\b},\hat{\c},\hat{\d},\hat{\e})=\0$. 
First, the fourth row indicates
that $\hat{\c}=\hat{\e}$, therefore, the last row leads to:
\begin{equation}
\bar{z}_i\hat{d}_i+\bar{s}_i\hat{e}_i = \bar{z}\hat{d}_i+\bar{s}_i\hat{c}_i = 0
\label{zdse}
\end{equation}
for each $i \in \{1, \ldots, p\}$.
Therefore, we can derive from (A3) that
\begin{equation*}
\hat{\d}^{\T}  \hat{\c} =0.
\label{dTcIsZero}
\end{equation*}
Actually, for each $i \in \{1, \ldots, p\}$,
either $\bar{z}_i$ or $\bar{s}_i$ is positive. 
Thus, if $\bar{z}_i > 0$, (A3) implies $\bar{s}_i = 0$, 
therefore we know $\hat{d}_i = 0$ due to (\ref{zdse});
Similarly, if $\bar{s}_i > 0$, (A3) implies $\bar{z}_i = 0$,
we know $\hat{c}_i = 0$ due to (\ref{zdse}).
From the second and third rows of (\ref{independent}), 
we have
\begin{equation}
\A_E \hat{\a} =\0, ~~~~~ \A_I \hat{\a} - \hat{\d}=\0.
\label{rows23}
\end{equation}
Multiplying $\hat{\a}^{\T}$ from the left of the first row of
(\ref{independent}) and using (\ref{rows23}), we have
\begin{equation*}
\hat{\a}^{\T} \nabla_{\x}^2 L(\bar{\v}) \hat{\a}
+ \hat{\a}^{\T} \A_E^{\T}  \hat{\b} - \hat{\a}^{\T} \A_I^{\T}  \hat{\c}
= \hat{\a}^{\T} \nabla_{\x}^2 L(\bar{\v}) \hat{\a} - \hat{\d}^{\T}  \hat{\c} 
=\hat{\a}^{\T} \nabla_{\x}^2 L(\bar{\v}) \hat{\a} =0.
\label{aIsZero1}
\end{equation*}
In view of (A2),  we conclude $\hat{\a}=\0$. Then,
it follows from (\ref{rows23}) that $\hat{\d}=\0$,
therefore, we know $\bar{s}_i \hat{c}_i = 0$ for each $i$ from (\ref{zdse}).
If $\bar{s}_i >0$, it holds $\hat{c}_i = 0$ for $i \notin I(\bar{\x})$.
On the other hand,
if $\bar{s}_i = 0$,  it holds that $i \in I(\bar{\x})$, 
so that the first row of (\ref{independent}) turns to be
$\A_E^{\T} \hat{\b} - \A_I^{\T}  \hat{\c} = \0$,
since $\hat{c}_i = 0$ for  $i \notin I(\bar{\x})$.
Consequently, it holds 
$\hat{\b}=\0$ and $\hat{c}_i = 0$ for $i \in I(\bar{\x})$
because of (A1).
As a result, we obtain $\hat{\c} =\0$, and we already know 
$\hat{\c} = \hat{\e}$ from the fourth row of (\ref{independent}).
This proves the theorem.
\hfill \qed
\end{proof}

\begin{remark} 
For $\v$ not close to the optimal solution, the nonsigularity
of the Jacobian of $\g$ is carefully discussed in \cite{yiy21}.
\end{remark}

\section{The interior-point algorithm with arc-search}\label{sec:algorithm}
 
Let ${\v}\BL{t} =({\x}\BL{t},{\y}\BL{t},{\w}\BL{t},{\s}\BL{t},{\z}\BL{t}) \in 
\R^n \times \R^m \times \R^{3p}$ be a function
of $t>0$ which is the solution of the {\it modified perturbed} 
KKT conditions $\g({\v}\BL{t}) = t  \g({\v \BL{1}})$ with 
nonnegative conditions $({\w}\BL{t}, {\s}\BL{t}, {\z}\BL{t}) \in \R_+^{3p}$
given as follows:
\begin{eqnarray}
\left[ \begin{array}{l}
\nabla_{\x} L({\v}\BL{t}) \\
(\A_E \x -\b_E)\BL{t} \\
(\A_I \x-\s-\b_I)\BL{t} \\
\nabla_{\s} L({\v}\BL{t})  \\
(\Z {\s}) \BL{t}
\end{array} \right]= \left[ \begin{array}{l}
 t \r_C \\ 
 t \r_E \\ 
 t \r_I \\ 
t \nabla_{\s} L({\v})  \\
t \Z {\s} 
\end{array} \right],
\label{KKTcurve}
\end{eqnarray}
where the current iterate ${\v}=({\x},{\y},{\w},{\s},{\z})
=({\x}\BL{1},{\y}\BL{1}, {\w}\BL{1}, {\s}\BL{1}, {\z}\BL{1})$.
Clearly, ${\v}\BL{t}$ defines a curve in $\R^n \times \R^m \times \R^{3p}$
that passes the current iterate point ${\v}\BL{1}$.
We denote the high dimensional curve defined by (\ref{KKTcurve}) as
$$C = \left\{{\v}\BL{t} \in \R^{n+m+3p} : t \in (0,1] \right\}.$$
Note that the right-hand-side of (\ref{KKTcurve}) goes to zeros 
when $t \to 0$, therefore, $\g({\v}\BL{t}) \to \0$ and ${\v}\BL{t}$ converges to a 
KKT point given by \eqref{KKT1}.

Since the calculation of ${\v}\BL{t}$ is very expensive,
the idea of the arc-search is to efficiently approximate the 
curve $C$ by using part of an ellipse and searching for optimizer
along the ellipse. We denote the ellipse by
\begin{equation}
{\cal E}=\lbrace \v\AN{\alpha}: 
\v\AN{\alpha} =
\vec{\a}\cos(\alpha)+\vec{\b}\sin(\alpha)+\vec{\c}, \alpha \in [0, 2\pi] \rbrace,
\label{ellipse}
\end{equation}
where $\vec{\a} \in \R^{n+m+3p}$ and $\vec{\b} \in \R^{n+m+3p}$ are 
the axes of the ellipse, and $\vec{\c} \in \R^{n+m+3p}$ is its center. 
The calculation of $\vec{\a}$, $\vec{\b}$, and $\vec{\c}$ can be 
avoided by using the analytical formulas given in Theorem \ref{ellipseSX}
\cite{yang09}. Denote
\begin{align*}
\dot{\v} = (\dot{\x}, \dot{\y}, \dot{\w}, \dot{\s},  \dot{\z}).
\end{align*} 
Taking  the derivative on both sides of (\ref{KKTcurve}), 
we get the linear systems of equations
\begin{equation}
\left[ \begin{array}{ccccc}
\H & \A_E^{\T} & -\A_I^{\T} & \0  & \0 \\
\A_E & \0 & \0 & \0 & \0  \\
\A_I & \0 & \0 & -\I  & \0 \\
\0  & \0  &  \I  &  \0  &  -\I  \\
\0 & \0  &\0 & \Z & \S
\end{array} \right]
\left[ \begin{array}{c}
\dot{\x} \\ \dot{\y} \\ \dot{\w} \\ \dot{\s}  \\ \dot{\z}
\end{array} \right]
= \left[ \begin{array}{l}
\r_C \\ 
\r_E \\ 
\r_I \\ 
{\w}-{\z}  \\
\Z {\s} 
\end{array} \right].
\label{firstOrder}
\end{equation}
The first-order derivative of the curve $\v\BL{t}$ at $t=1$ along 
$C$ is denoted by $\dot{\v}$.
Let  $\sigma \in [0,1]$ be the centering parameter (see \cite{wright97}).
The second-order derivative 
$\ddot{\v} = (\ddot{\x}, \ddot{\y}, \ddot{\w}, \ddot{\s}, \ddot{\z})$ 
at $t=1$ along the curve is defined
as the solution of the following linear systems of equations:
{\footnotesize
\begin{eqnarray}
\left[ \begin{array}{ccccc}
\H  & \A_E^{\T} & -\A_I^{\T} & \0  & \0 \\
\A_E & \0 & \0 & \0 & \0  \\
\A_I & \0 & \0 & -\I  & \0 \\
\0  & \0  &  \I  &  \0  &  -\I  \\
\0 & \0 & \0  & \Z & \S
\end{array} \right]
\left[ \begin{array}{c}
\ddot{\x} \\ \ddot{\y} \\ \ddot{\w} \\ \ddot{\s} \\ \ddot{\z}
\end{array} \right]
=  \left[ \begin{array}{c}
-(\nabla_{\x}^3 f(\x))\dot{\x} \dot{\x} \\
\0 \\
\0 \\
\0  \\
-2\Z \dot{\s} 
\end{array} \right]
\approx
 \left[ \begin{array}{c}
\0 \\
\0 \\
\0 \\
\0  \\
-2\dot{\Z} \dot{\s}  +\sigma {\mu} \e
\end{array} \right].
\label{secondOrder}
\end{eqnarray}
}
We add a centering item $\sigma {\mu} \e$ to the last element in right 
hand side, which is the same strategy used in \cite{yang20}.
This modification assures that a substantial segment of 
the ellipse satisfies the requirement of $(\s,\z)>\0$, thereby  
assures that the step size along the ellipse is greater than zero. 
Our experience in \cite{yiy21} shows that the computation of 
$(\nabla_{\x}^3 f(\x))\dot{\x} \dot{\x}$ is very expensive.
Therefore, to have an efficient algorithm, we omit this higher 
order term in the rest discussion. We show that this modification
leads to an algorithm that converges in polynomial time. It is 
worthwhile to mention that, according to (\ref{secondOrder}),
$\ddot{\v} = (\ddot{\x}, \ddot{\y}, \ddot{\w},\ddot{\s},\dot{\z})$
is a function of $\sigma$, i.e., $\ddot{\v}$ should be written as
$\ddot{\v}(\sigma)$. But we use $\ddot{\v}$ most time
when no confusion is introduced.

Using $\dot{\v}$ and $\ddot{\v}$, we can approximate 
$C$ at $t=1$ by an ellipse (\ref{ellipse}) 
that has the explicit form given in Theorem \ref{ellipseSX}.
We should emphasize that we use $t$ to denote the curve $C$
and ${\v}\BL{t}$ passes ${\v}$ at $t=1$,
while we use the angle $\alpha$ to express an ellipse ${\cal E}$
and ${\v}\AN{\alpha}$ passes $\v$ at $\alpha = 0$, therefore,
${\v}\BL{1} = {\v}\AN{0} = {\v}$.

\begin{theorem}[\cite{yang11}]\label{theorem:ellip}
Assume that an ellipse ${\cal E}$ 
of form \textrm{(\ref{ellipse})} passes through 
the current iterate ${\v}$ at $\alpha=0$, let the first and second 
order derivatives at $\alpha=0$ be $\dot{\v}$ 
and $\ddot{\v}$ which are defined by (\ref{firstOrder}) 
and (\ref{secondOrder}), respectively. Then the curve ${\v}\AN{\alpha}$
depends on the selection of $\sigma$ and
${\v}\AN{\alpha,\sigma} = ({\x}\AN{\alpha,\sigma}, {\y}\AN{\alpha,\sigma}, 
{\w}\AN{\alpha,\sigma}, {\s}\AN{\alpha,\sigma}, {\z}\AN{\alpha,\sigma})$ 
of ${\cal E}$ is given by
\begin{align}
{\v}\AN{\alpha,\sigma} = {\v} - \dot{\v}\sin(\alpha)+\ddot{\v}(\sigma)(1-\cos(\alpha)),
\label{vAlpha}
\end{align}
or 
\begin{align}
\left[ \begin{array}{c}
{\x}^{k+1} \\ {\y}^{k+1} \\ {\w}^{k+1} \\ {\s}^{k+1} \\ {\z}^{k+1}
\end{array} \right]
= \left[ \begin{array}{c}
{\x}^{k} \\ {\y}^{k} \\ {\w}^{k} \\ {\s}^{k} \\ {\z}^{k}
\end{array} \right]
-\left[ \begin{array}{c}
\dot{\x} \\ \dot{\y} \\ \dot{\w} \\ \dot{\s} \\ \dot{\z}
\end{array} \right] \sin(\alpha^k)
+\left[ \begin{array}{c}
\ddot{\x}(\sigma) \\ \ddot{\y}(\sigma) \\ \ddot{\w}(\sigma) \\ \ddot{\s}(\sigma) \\ \ddot{\z}(\sigma)
\end{array} \right] (1-\cos(\alpha^k)).
\label{allAlpha}
\end{align}
\label{ellipseSX}
\end{theorem}

We would like to emphasize that the second derivatives (therefore the ellipse)
are functions of both $\alpha$ and $\sigma$ which we will carefully select 
simultaneously in every iteration $k$.
The following lemma can be used to simplify the computation of (\ref{allAlpha}).
\begin{lemma}[\cite{yiy21}]\label{eqwz}
If ${\v}$ satisfies ${\w}={\z}$, then ${\w}\AN{\alpha}={\z}\AN{\alpha}$ 
holds for any $\alpha \in \R$. 
\end{lemma}
\begin{proof}
The proof is straightforward and therefore is omitted.
\hfill \qed
\end{proof}

For numerical stability, we need that the Jacobian stays away
from singularity. Therefore, we make the following assumption.

\vspace{0.05in}\noindent
{\bf Assumption:}
\begin{itemize}
\item[(A3')] $\Z^k > \0$ and $\S^k > \0$ are bounded below and
away from zeros for all $k$ iterations until the program is terminated.
\end{itemize}
It will be clear that this assumption is also important in the convergence
analysis. 

As discussed in \cite{yang18} and \cite{yang21}, it is a good 
strategy to simultaneously select the step size $\alpha$ and the centering
paramenter $\sigma$ whenever it is possible. To this end, we 
should express $\ddot{\v}$ explicitly as a function of $\sigma$. This can 
be done by solving two linear systems of equations:

\begin{eqnarray}
\left[ \begin{array}{ccccc}
\H  & \A_E^{\T} & -\A_I^{\T} & \0  & \0 \\
\A_E & \0 & \0 & \0 & \0  \\
\A_I & \0 & \0 & -\I  & \0 \\
\0  & \0  &  \I  &  \0  &  -\I  \\
\0 & \0 & \0  & \Z & \S
\end{array} \right]
\left[ \begin{array}{c}
\p_{\x} \\ \p_{\y} \\ \p_{\w} \\ \p_{\s} \\ \p_{\z}
\end{array} \right]
=   \left[ \begin{array}{c}
\0 \\
\0 \\
\0 \\
\0  \\
\mu \e 
\end{array} \right],
\label{secondOrderP}
\end{eqnarray}
and
\begin{eqnarray}
\left[ \begin{array}{ccccc}
\H  & \A_E^{\T} & -\A_I^{\T} & \0  & \0 \\
\A_E & \0 & \0 & \0 & \0  \\
\A_I & \0 & \0 & -\I  & \0 \\
\0  & \0  &  \I  &  \0  &  -\I  \\
\0 & \0 & \0  & \Z & \S
\end{array} \right]
\left[ \begin{array}{c}
\q_{\x} \\ \q_{\y} \\ \q_{\w} \\ \q_{\s} \\ \q_{\z}
\end{array} \right]
=  \left[ \begin{array}{c}
\0 \\
\0 \\
\0 \\
\0  \\
-2\dot{\Z} \dot{\s}
\end{array} \right],
\label{secondOrderQ}
\end{eqnarray}
denoting $\p=(\p_{\x}, \p_{\y}, \p_{\w}, \p_{\s}, \p_{\z})$ and
$\q=(\q_{\x}, \q_{\y}, \q_{\w}, \q_{\s}, \q_{\z})$, then we have
$\ddot{\v}= \p \sigma + \q$. Solving (\ref{firstOrder}), (\ref{secondOrderP}),
and (\ref{secondOrderQ}) is equivalent to solve linear systems of
equations $\A \d_i = \b_i$ for $i=1,2,3$ with the same $\A$ 
but different $\b_i$. Therefore we can use the same decomposition 
of $\A$ three times as indicated in \cite{wright97}, which
justifies the strategy of splitting (\ref{secondOrder}) into (\ref{secondOrderP})  
and (\ref{secondOrderQ}).

For the convex programming problem (\ref{CP}),
it is well-known that a vector 
$\bar{\v}=(\bar{\x},\bar{\y},\bar{\w},\bar{\s},\bar{\z})$ 
that meets the KKT conditions is the optimal solution. Therefore, 
we need to show that the proposed algorithm for problem (\ref{CP})
will generate a sequence $\v^k$ such that it approaches to an point
$\bar{\v}$ that meets the approximate KKT conditions:
\begin{itemize}
\item[(C1).] $(\bar{\r}_C, \bar{\r}_E, \bar{\r}_I) \le \epsilon$.
\item[(C2).] $(\w^k,\s^k,\z^k) >\0$ before the program terminates
at $(\bar{\w},\bar{\s},\bar{\z}) \ge 0$.
\item[(C3).] $\bar{\mu} \le \epsilon$ (given $(\bar{\s},\bar{\z}) \ge \0$, this 
is equivalent to $\bar{\z}^{\T} \bar{\s} \le p \epsilon$).
\end{itemize}

In addition to the approximate KKT conditions, we will restrict the search 
in an interior point region given as follows:
\begin{equation}
{\mathcal F} =\lbrace(\s, \z): \hspace{0.1in} (\s, \z) >\0,
\hspace{0.1in} s_i^kz_i^k \ge \theta \mu_k \rbrace, 
\label{infeasible1} 
\end{equation}
where $\theta \in (0,1)$ is a constant.
Therefore, this imposes one more condition on $\v^k$:
\begin{itemize}
\item[(C4).] 
\begin{equation}
\S^k \z^k = \Z^k \s^k \ge \theta \mu_k \e.
\label{cond4}
\end{equation}
\end{itemize}

The following proposition shows that searching along the ellipsoidal arc does
improve the objective function and the feasibility of the constraints.
Moreover, if $\alpha^k$ is bounded below and away from zero for
all $k$, The above-mentioned Condition (C1) will hold.

\begin{proposition}
Denote $\nu_k=\prod_{j=0}^{k-1} (1-\sin(\alpha^j))$. We have the
following formulas.
\begin{subequations}
\begin{align}
\r_C^{k+1} = 
\r_C^{k} (1-\sin(\alpha^{k})) = \cdots 
= \r_C^0 \prod_{j=0}^{k} (1-\sin(\alpha^j))
= \r_C^0 \nu_k, \label{a}
\\
\r_E^{k+1} = \r_E^{k} (1-\sin(\alpha^{k})) = \cdots 
= \r_E^0 \prod_{j=0}^{k} (1-\sin(\alpha^j))
= \r_E^0 \nu_k. \label{b}
\\
\r_I^{k+1} = \r_I^{k} (1-\sin(\alpha^{k})) = \cdots 
= \r_I^0 \prod_{j=0}^{k} (1-\sin(\alpha^j))
= \r_I^0 \nu_k. \label{c}
\end{align}
\label{errorUpdate}
\end{subequations}
\label{basic}
\end{proposition}
\begin{proof}
Using (\ref{residuals}), (\ref{allAlpha}), and the first lines
of (\ref{firstOrder}) and (\ref{secondOrder}), we have
\begin{eqnarray}
\r_C^{k+1}-\r_C^{k} & = &
\H (\x^{k+1}-\x^k) + \A_E(\y^{k+1}-\y^k)-\A_I(\w^{k+1}-\w^k)
\nonumber \\
 &  = & \H [-\dot{\x}\sin(\alpha)+\ddot{\x}(1-\cos(\alpha))]
\nonumber \\
 &  & + \A_E[-\dot{\y}\sin(\alpha)+\ddot{\y}(1-\cos(\alpha))]
\nonumber \\
 &  & -   \A_I[-\dot{\w}\sin(\alpha)+\ddot{\w}(1-\cos(\alpha))]
\nonumber \\
 &  = & -[\H \dot{\x}+\A_E\dot{\y}-\A_I\dot{\w}]\sin(\alpha)
\nonumber \\
 &  & +[\H \ddot{\x}+\A_E\ddot{\y}-\A_I\ddot{\w}](1-\cos(\alpha))
\nonumber \\
 &  = & -\r_C^{k}\sin(\alpha).
\end{eqnarray}
This shows $\r_C^{k+1} = \r_C^{k} (1-\sin(\alpha^{k}))$. 
Following a similar argument proves (\ref{b}) and (\ref{c}).
\hfill \qed
\end{proof}

\begin{remark}
This proposition indicates that searching along the ellipsoidal curve will
improve the objective function and feasibility at the same rate in every
iteration. The larger the $\alpha$ is, the faster the improvement will be.
\end{remark}

If $\nu_k=0$, then $\r_E= \0$ and $\r_I= \0$, Problem 
(\ref{CP}) is reduced to a feasible convex programming problem, 
for which a feasible interior point algorithm such as \cite{kz93}
should be a more appropriate choice. Therefore, we 
make the following assumption as below.

\vspace{0.05in}\noindent
{\bf Assumption:}
\begin{itemize}
\item[(A4)] $\nu_k >0$ for all $k>0$.
\end{itemize}

To meet the positiveness requirement of Condition (C2), we adopt the
strategy described in \cite{yang18}. Let $\rho \in (0,1)$ be a constant, and
\begin{equation}
\underline{s}_k = \displaystyle\min_i s^k_i, \hspace{0.1in}
\underline{z}_k = \displaystyle \min_j z^k_j.
\label{ijmath}
\end{equation}
Denote $\phi_k$ and $\psi_k$
such that 
\begin{align}
\phi_k = \min \{ \rho \underline{s}_k , \nu_k \}, \hspace{0.1in}
\psi_k = \min \{ \rho \underline{z}_k, \nu_k \}. 
\label{phikpsi}
\end{align}
It is clear that 
\begin{subequations}
\begin{align}
\0< \phi_k \e \le \rho \s^k, \hspace{0.1in} \0< \phi_k \e \le \nu_{k}\e, 
\label{phi} \\
\0< \psi_k \e \le  \rho \z^k, \hspace{0.1in} \0< \psi_k \e \le  \nu_{k}\e.
\label{psi}
\end{align}
\end{subequations}

Positivity of $\s(\sigma_k, \alpha_k)$ and $\z(\sigma_k,\alpha_k)$ is
guaranteed if $(\s^0,\z^0)>\0$ and the following conditions hold.
\begin{eqnarray}
\s^{k+1} & = & \s(\sigma_k, \alpha_k) 
= \s^k-\dot{\s} \sin(\alpha_k) +\ddot{\s}(1-\cos(\alpha_k)) \nonumber \\
& = & {\p}_{\s} (1-\cos(\alpha_k)) \sigma_k +
[\s^k-\dot{\s} \sin(\alpha_k)+ \q_{\s} (1-\cos(\alpha_k))]
\nonumber \\
& := & \a_s(\alpha_k) \sigma_k  +\b_s(\alpha_k)   \ge \phi_k \e. 
\label{posi1}
\end{eqnarray}
\begin{eqnarray}
\z^{k+1} & = & \z(\sigma, \alpha_k) 
= \z^k-\dot{\z} \sin(\alpha_k) +\ddot{\z}(1-\cos(\alpha_k)) \nonumber \\
& = &  {\p}_{\z} (1-\cos(\alpha_k)) \sigma_k  +
[\z^k-\dot{\z} \sin(\alpha_k)+ \q_{\z} (1-\cos(\alpha_k))]
\nonumber \\
& := & \a_z(\alpha_k) \sigma_k  +\b_z(\alpha_k)  \ge \psi_k \e. 
\label{posi2}
\end{eqnarray}

\begin{remark}
Conditions of (\ref{posi1}) and (\ref{posi2}) will be enforced in
the algorithm. Given $\phi_k$ and $\psi_k$ as calculated in
(\ref{phikpsi}), the corresponding $\alpha_k$ and $\sigma_k$
that meet (\ref{posi1}) and (\ref{posi2}) will be calculated by
the formulas (\ref{alphaS})-(\ref{case7b}) and a process described in Algorithm
\ref{bisectionSearch}.
\end{remark}

If $\s^{k+1}= \s^k-\dot{\s} \sin(\alpha_k) 
+\ddot{\s}(1-\cos(\alpha_k)) \ge \rho \s^k$ holds, from (\ref{phi}), we have
$\s^{k+1}\ge \phi_k \e$. Therefore, inequality (\ref{posi1}) will hold if
the following inequality holds
\begin{equation}
(1-\rho ) \s^k -\dot{\s} \sin(\alpha_k) +\ddot{\s}(1-\cos(\alpha_k)) \ge \0.
\label{posi3}
\end{equation} 
In view of (\ref{psi}), inequality (\ref{posi2}) will hold if the following inequality holds
\begin{equation}
(1-\rho ) \z^k -\dot{\z} \sin(\alpha_k) +\ddot{\z}(1-\cos(\alpha_k)) \ge \0.
\label{posi4}
\end{equation}
Inequalities (\ref{posi3}) and (\ref{posi4}) hold for some $\alpha_k>0$ 
bounded below and away from zero because $(1-\rho ) \s^k >\0$ and
$(1-\rho ) \z^k >\0$ is bounded below and away from zero
due to Assumption (A3'). 

The following proposition follows immediately from the above discussion.
\begin{proposition}
There is an $\alpha_k>0$ bounded below and away from zero such that
$(\s^{k+1}, \z^{k+1})>\0$ for all iteration $k$.
\label{secondBound}
\end{proposition}

We will also need the following results in the rest discussions.

\begin{lemma}
Let $\dot{\v}$ and $\ddot{\v}$ be defined as in (\ref{firstOrder})
and (\ref{secondOrder}), and let ${\p}$ and ${\q}$ be defined 
as in (\ref{secondOrderP}) and (\ref{secondOrderQ}). Then the following relations hold.
\begin{subequations}
\begin{gather}
\ddot{\x}^{\Tr} \A_I^T \ddot{\w} = \ddot{\x}^{\Tr} \A_I^T \ddot{\z} 
=\ddot{\x}^{\Tr} \H  \ddot{\x} \ge 0,  \hspace{0.1in}
  \ddot{\s}^{\Tr}  \ddot{\z}  \ge 0. \label{ineqA} \\
\p_{\x}^{\Tr} \A_I^T \p_{\z} 
=\p_{\x}^{\Tr} \H  \p_{\x} \ge 0,  \hspace{0.1in}
  \p_{\s}^{\Tr}  \p_{\z}  \ge 0.  \label{ineqB} \\
\q_{\x}^{\Tr} \A_I^T \q_{\z} 
=\q_{\x}^{\Tr} \H  \q_{\x} \ge 0,  \hspace{0.1in}
  \q_{\s}^{\Tr}  \q_{\z}  \ge 0.   \label{ineqC} 
\end{gather}
\label{cross}
\end{subequations}
\end{lemma}
\begin{proof}
Pre-multiplying $\ddot{\x}^{\Tr}$ in the first line of (\ref{secondOrder})
gives
\[
\ddot{\x}^{\Tr} \H  \ddot{\x}+\ddot{\x}^{\Tr} \A_E^{\Tr}\ddot{\y}
-\ddot{\x}^{\Tr} \A_I^{\Tr}\ddot{\w} =0.
\]
From the second line of (\ref{secondOrder}), we have
$\ddot{\x}^{\Tr} \A_E^{\Tr} =\0$. Therefore,
$\ddot{\x}^{\Tr} \H  \ddot{\x}=\ddot{\x}^{\Tr} \A_I^{\Tr}\ddot{\w} \ge 0$
because $\H $ is positive definite.

Pre-multiplying $\ddot{\z}^{\Tr}$ in the third line of (\ref{secondOrder})
gives
\[
\ddot{\z}^{\Tr}  \A_I \ddot{\x} -\ddot{\z}^{\Tr}\ddot{\s} =0.
\]
Therefore, using $\ddot{\z}=\ddot{\w}$, 
$\ddot{\z}^{\Tr}\ddot{\s} = \ddot{\z}^{\Tr}  \A_I \ddot{\x}
=\ddot{\x}^{\Tr} \H \ddot{\x} \ge 0$. Similarly, we can prove (\ref{ineqB})
and (\ref{ineqC}) using (\ref{secondOrderP}) and (\ref{secondOrderQ}).
\hfill \qed
\end{proof}

In addition, we need two simple sinusoidal identities in our proofs.
\begin{lemma}
\begin{subequations}
\begin{gather}
\sin^2(\alpha)-2(1-\cos(\alpha))=-(1-\cos(\alpha))^2, \label{subeq}
\\ 
\sin^2(\alpha) \ge 1-\cos(\alpha) \ge \frac{1}{2} \sin^2(\alpha).
\label{subineq}
\end{gather}
\end{subequations}
\label{sincos}
\end{lemma}
\begin{proof}
The proof is straightforward, therefore it is omitted.
\hfill \qed
\end{proof}

The above two lemmas, together with  (\ref{firstOrder}), (\ref{secondOrder}), 
(\ref{allAlpha}), (\ref{secondOrderP}), and (\ref{secondOrderQ}),
will be used to calculate the value $\mu(\alpha)$.

\begin{proposition}
Let ${\alpha}_k$ be the step size at $k$th iteration for 
$\s(\sigma_k, \alpha_k)$, $\z(\sigma_k, \alpha_k)$ defined in 
Theorem \ref{ellipseSX}. Then, the updated duality measure after the
iteration of $k$ can be expressed as 
\begin{equation}
\mu_{k+1} := \mu(\sigma_k,{\alpha}_k) =
\frac{1}{p} \left[ a_u(\alpha_k)  \sigma_k + b_u(\alpha_k) \right],
\label{usigmaAlpha}
\end{equation}
where 
\[a_u(\alpha_k)=p \mu_k (1-\cos(\alpha_k)) - (\dot{\z}^{\T} {\p}_{\s}
+\dot{\s}^{\T} {\p}_{\z}) 
\sin(\alpha_k)(1-\cos(\alpha_k))\] 
and 
\[
b_u(\alpha_k)=p \mu_k (1-\sin(\alpha_k)) - [\dot{\z}^{\T} \dot{\s} (1-\cos(\alpha_k))^2
+(\dot{\s}^{\T}\q_{\z} + \dot{\z}^{\T} \q_{\s}) \sin(\alpha_k)(1-\cos(\alpha_k))]
\]
are coefficients which are functions of $\alpha_k$. Moreover
$a_u(\alpha_k)={O}(p \mu_k \sin^2(\alpha))$ and 
$b_u(\alpha_k)={O}(p \mu_k (1-\sin(\alpha_k)))$.
\label{simple2}
\end{proposition}
\begin{proof} 
The proof is similar to the one of \cite{yang18} and therefore omitted.
\hfill \qed
\end{proof}

\begin{remark}
Proposition \ref{simple2} shows that if a positive $\alpha$ small enough, the duality gap
is guaranteed to decrease.
\end{remark}

The following proposition assures that Condition (C4) will hold.

\begin{proposition}
There is an $\alpha_k$ bounded below and away
from zero for all $k$ such at (\ref{cond4}) holds, i.e.,
\begin{eqnarray}
& & s_i^{k+1} z_i^{k+1} \nonumber \\
& \ge &   \theta \mu_k (1 -\sin(\alpha_k) ) 
+\sigma_k \mu_k (1-\cos(\alpha_k)) 
  \nonumber \\ 
&  & - [\ddot{s}_i^k \dot{z}_i^k+\dot{s}_i^k \ddot{z}_i^k] \sin(\alpha_k) 
(1-\cos(\alpha_k))+ (\ddot{s}_i^k \ddot{z}_i^k-\dot{x}_i^k \dot{s}_i^k)
(1-\cos(\alpha_k))^2.
\label{thirdBound1}
\end{eqnarray}
\label{thirdBound}
\end{proposition}
\begin{proof}
The proof is similar to the one of \cite{yang18}, therefore is omitted.
\hfill \qed
\end{proof}

Propositions \ref{basic}, \ref{secondBound}, \ref{simple2}, and \ref{thirdBound}
indicate that an arc-search strategy will generate a sequence of iterates
that will eventually meet the approximated KKT
conditions, plus Condition (C4) defined by (\ref{cond4}). 
Therefore, we propose the following algorithm.

\begin{algorithm}\label{algorithm:arc-search} {\bf (an infeasible arc-search interior-point algorithm
	)} \newline
\upshape 
\indent Parameters: $\theta \in (0,1)$, and $\epsilon>0$.
\newline\indent Initial point: $\v^0 = (\x^0,\y^0,\w^0,\s^0,\z^0)$, 
$(\w^0,\s^0,\z^0) \in \R_{++}^{3p}$, and $\w^0=\z^0$.
\newline
\indent {\bf for} iteration $k=0,1,2,\ldots$
\begin{itemize}
	\item[] Step 1:  If $\| \g(\v^k) \| \le \epsilon$, stop.
	\item[] Step 2: Calculate $\nabla_{\x} L(\v^k)$, $\A_E \x^k-\b_E$, 
    $\A_I \x^k-\s^k -\b_I$, and $\H=\nabla_{\x}^2 L(\v^k)$.
	\item[] Step 3: Solve (\ref{firstOrder}) to get
	$\dot{\v}^k = (\dot{\x}^k,\dot{\y}^k,\dot{\w}^k,\dot{\s}^k,\dot{\z}^k)$.
	\item[] Step 4: Calculate $\dot{\Z} \dot{\s}$, $\p$, and $\q$.
	\item[] Step 5: Select $\sigma_k \in [0,1]$ and $\alpha_k > 0$ such that
    $\ddot{\v}^k=\p \sigma_k + \q$, $\v^{k+1} = \v^k\AN{\sigma_k,\alpha_k} = \v^k 
	- \dot{\v}^k \sin(\alpha_k) + \ddot{\v}^k (1-\cos(\alpha_k))$, 
    $\mu_{k+1} < \mu_k$, $(\s^{k+1},\w^{k+1},\z^{k+1})>\0$, and     
    $\Z^{k+1}\s^{k+1} \ge \theta \mu_{k+1} \e$.
	\item[] Step 6:  $k+1 \rightarrow k$ and go back to Step 1.
\end{itemize}
\indent\indent {\bf end (for)} 
\hfill \qed
\label{mainAlgo1}
\end{algorithm}

The next section shows that the proposed algorithm converges in
polynomial time.

\section{Convergence analysis}\label{sec:convergence}

In view of Propositions \ref{basic}, \ref{secondBound}, \ref{simple2}, 
and \ref{thirdBound}, the algorithm will generate a series 
of $\alpha_k$ which is bounded below and away from zero before the
algorithm terminates. Therefore, there exist a constant $\rho \in (0,1)$ 
satisfying $\rho \ge (1- \sin(\alpha_k))$ for all $k \ge 0$. Define
\begin{equation}
\beta_k = \frac{\min \{ \underline{s}_k , \underline{z}_k \}}{\nu_k} \ge 0,
\label{betak}
\end{equation}
and 
\begin{equation}
\beta = \displaystyle\inf_{k} \{ \beta_k \} \ge 0.
\label{beta}
\end{equation}

The next lemma is obtained in \cite{yang18} and it shows that $\beta$ is bounded below from zero. 

\begin{lemma}[\cite{yang18}]
Assuming that $\rho \in (0,1)$ is a constant and for all 
$k \ge 0$, $\rho \ge (1- \sin(\alpha_k))$. Let $\underline{s}_0=\min_i s_i^0$
and $\underline{z}_0=\min_i z_i^0$. Then, we have 
$\beta \ge \min \{ \underline{s}_0 , \underline{z}_0, 1 \}$.
\label{betaGt0}
\end{lemma} 

Let  $\D= \S^{\frac{1}{2}} \Z^{-\frac{1}{2}}=\diag(D_{ii})$. The
next result can be derived using the same method in \cite{yang18}. 

\begin{lemma} [\cite{yang18}]
For Algorithm \ref{mainAlgo1}, there is a constant $C_1$ independent 
of $n$ and $m$ such that for $\forall i \in \{ 1, \ldots, n \}$
\begin{equation}
(D^{k}_{ii})^{-1}  \nu_k = \nu_k \sqrt{\frac{z_i^k}{s_i^k}}
\le C_1 \sqrt{n \mu_k}, \hspace{0.1in}
D_{ii}^k  \nu_k = \nu_k \sqrt{\frac{s_i^k}{z_i^k}}
\le C_1 \sqrt{n \mu_k}.
\end{equation}
\label{cond1}
\end{lemma}

Note that $\bar{\w}=\bar{\z}$. Let 
$(\bar{\x}, \bar{\y}, \bar{\w}, \bar{\s}, \bar{\z})$ be an 
optimal solution satisfying 
\begin{subequations}
\begin{gather}
\H  \bar{\x}+ \A_E^{\T}\bar{\y}  -\A_I^{\T} \bar{\z} = \0, 
\label{fa} \\
\A_E \bar{\x} = \b_E, \label{fb} \\
\A_I \bar{\x} -\bar{\s} = \b_I. \label{fc} 
\end{gather}
\label{feasible}
\end{subequations}
The following assumption means that the distance between
the initial point and the optimal solution is bounded, which is reasonable.
\newline
\newline
{\bf Assumption:}
\begin{itemize}
\item[(A5)] {\it There is a big constant $M$ which is independent 
to the problem size $n$ and $m$ such that the distance between
initial point $(\s^0, \z^0)$ and $(\bar{\s}, \bar{\z})$ is smaller
than $M$, i.e.,
$\| \left( \s^0-\bar{\s}, \z^0-\bar{\z} \right) \| < M$. }
\end{itemize}

\begin{lemma}
For an optimal solution $(\bar{\x}, \bar{\y}, \bar{\w}, \bar{\s}, \bar{\z})$
of (\ref{CP}) that meets (\ref{feasible}), the following linear systems of equations
\begin{equation}
\left[ \begin{array}{cccc}
\H & \A_E^{\T} & - \A_I^{\T} & \0 \\
\A_E & \0  & \0  & \0 \\
\A_I & \0  & \0  & -\I \\
\0  & \0  &  \S  & \Z
\end{array} \right]
\left[ \begin{array}{c}
\delta \x \\ \delta \y  \\ \delta \z  \\ \delta \s 
\end{array} \right]
=\left[ \begin{array}{c}
\0 \\ \0 \\ \0 \\
\z\s -\nu_k \s(\z^0-\bar{\z}) -\nu_k \z(\s^0-\bar{\s})
\end{array} \right].
\label{deltaEq}
\end{equation}
has a solution given by
\begin{equation}
\left[ \begin{array}{c}
\delta \x \\ \delta \y  \\ \delta \z  \\ \delta \s 
\end{array} \right]
=\left[ \begin{array}{c}
\dot{\x} - \nu_k (\x^0 - \bar{\x}) )  \\
\dot{\y} - \nu_k (\y^0 - \bar{\y}) )  \\
\dot{\z} - \nu_k (\z^0 - \bar{\z}) )  \\
\dot{\s} - \nu_k (\s^0 - \bar{\s}) ) 
\end{array} \right].
\label{deltaSolution}
\end{equation}
\label{dettaEqLemma}
\end{lemma}
\begin{proof}
Using the first row of (\ref{firstOrder}), (\ref{a}), (\ref{defineG}), 
and (\ref{fa}), we have
\[
\H \dot{\x} + \A_E^{\T} \dot{\y} -\A_I^{\T} \dot{\z}
=\r_C^k = \nu_k \r_C^0 = \nu_k (\H \x^0 + \A_E^{\T} \y^0 - \A_I^{\T} \z^0
- \H \bar{\x}-\A_E^{\T} \bar{\y}+\A_I^{\T} \bar{\z}).
\]
This gives
\begin{equation}
\H [\dot{\x} -\nu_k ( \x^0-\bar{\x})]
+\A_E^{\T} [ \dot{\y} -\nu_k (\y^0 -\bar{\y})]
-\A_I^{\T} [\dot{\z} -\nu_k ( \z^0-\bar{\z})]=\0,
\label{1stRow}
\end{equation}
which is the first row of (\ref{deltaEq}).
Using the second row of (\ref{firstOrder}), (\ref{b}), (\ref{defineG}), 
and (\ref{fb}), we have
\[
 \A_E  \dot{\x} =\r_E^k = \nu_k \r_E^0
= \nu_k (\A_E \x^0 - \b_E)= \nu_k (\A_E \x^0 - \A_E \bar{\x}).
\]
This gives
\begin{equation}
 \A_E  [\dot{\x} -\nu_k ( \x^0-\bar{\x})]=\0,
\label{2ndRow}
\end{equation}
which is the second row of (\ref{deltaEq}).
Using the third row of (\ref{firstOrder}), (\ref{c}), (\ref{defineG}), 
and (\ref{fc}), we have
\[
 \A_I \dot{\x} -\dot{\s}  =\r_I^k = \nu_k \r_I^0
= \nu_k (\A_I \x^0 - \s^0 - \b_I)= \nu_k [\A_I (\x^0-\bar{\x}) - (\s^0 - \bar{\s})].
\]
This gives
\begin{equation}
 \A_I [\dot{\x} -\nu_k ( \x^0-\bar{\x})]-[\dot{\s}-\nu_k (\s^0 - \bar{\s})]  =\0,
\label{3rdRow}
\end{equation}
which is the third row of (\ref{deltaEq}).
Finally using the last row of (\ref{firstOrder}), we have
\begin{eqnarray}
&&  \S[\dot{\z}-\nu_k (\z^0 - \bar{\z})]+
\Z[\dot{\s}-\nu_k (\s^0 - \bar{\s})] 
\nonumber \\
& = & 
( \S \dot{\z} +\Z \dot{\s})-\nu_k [ \Z(\s^0 - \bar{\s})+\S(\z^0 - \bar{\z})]
\nonumber \\
& = & \Z\s - \nu_k [ \Z(\s^0 - \bar{\s})] -\nu_k [ \S(\z^0 - \bar{\z})]
\label{4eqnarray}
\end{eqnarray}
which is the last row of (\ref{deltaEq}). This completes the proof.
\hfill \qed
\end{proof}

Let 
\begin{equation}
\left[ \begin{array}{c}
\r^1 \\ \r^2  \\ \r^3 
\end{array} \right]
=\left[ \begin{array}{c}
\Z\s  \\
- \nu_k [ \Z(\s^0 - \bar{\s})]  \\
- \nu_k [ \S(\z^0 - \bar{\z})]   
\end{array} \right],
\label{splitR}
\end{equation}
and let $(\delta \x^i, \delta \y^i, \delta \z^i,  \delta \s^i)$, for 
$i=1,2,3$, be the solution of
\begin{equation}
\left[ \begin{array}{cccc}
\H  & \A_E^{\T} & - \A_I^{\T} & \0 \\
\A_E & \0  & \0  & \0 \\
\A_I & \0  & \0  & -\I \\
\0  & \0  &  \S  & \Z
\end{array} \right]
\left[ \begin{array}{c}
\delta \x^i \\
\delta {\y}^i \\
\delta {\z}^i \\
\delta \s^i
\end{array} \right] 
=
\left[ \begin{array}{c}
0 \\  0  \\  0 \\  \r^i 
\end{array} \right].
\label{matrixSplit}
\end{equation}

We will use the following result in our analysis.

\begin{lemma}
The solutions of (\ref{deltaEq}) and (\ref{matrixSplit}) meet
the following relations.
\begin{subequations}
\begin{gather}
\delta \x = \delta \x^1+\delta \x^2+\delta \x^3 
=\dot{\x} - \nu_k (\x^0-\bar{\x}), \label{deltaX} \\
\delta \y = \delta \y^1+\delta \y^2+\delta \y^3 
=\dot{\y} - \nu_k (\y^0-\bar{\y}) \label{deltaL}, \\
\delta \z = \delta \z^1+\delta \z^2+\delta \z^3 
=\dot{\z} - \nu_k (\z^0-\bar{\z}) \label{deltaZ}, \\
\delta \s = \delta \s^1+\delta \s^2+\delta \s^3 
=\dot{\s} - \nu_k (\s^0-\bar{\s}). \label{deltaS}
\end{gather}
\label{deltaXYZ}
\end{subequations} 
Moreover, $(\D^{-1} \delta \z^i)^{\Tr} (\D \delta \s^i) \ge 0$ 
holds, for $i=1,2,3$.
\label{inEq1}
\end{lemma}
\begin{proof}
The first claim follows immediately from the linear systems of equations
(\ref{deltaEq}) and (\ref{matrixSplit}). The second claim is 
equivalent to $(\delta \z^i)^{\Tr} (\delta \s^i) \ge 0$.
Pre-multiplying $(\delta \z^i)^{\Tr}$ in the third row of 
(\ref{matrixSplit}) yields
\begin{equation}
(\delta \z^i)^{\Tr} \A_I \delta \x ^i-(\delta \z^i)^{\Tr}\delta \s^i=0.
\label{tmp1}
\end{equation}
Rewriting the first row of (\ref{matrixSplit}) yields
\begin{equation}
\A_I^{\Tr} \delta \z^i=\H \delta \x ^i +\A_E^{\T}\delta \y^i .
\label{tmp2}
\end{equation}
Substituting (\ref{tmp2}) into (\ref{tmp1}) and using the second
row of (\ref{matrixSplit}) yields
\begin{eqnarray}
& & (\delta \z^i)^{\Tr} \A_I \delta \x^i-(\delta \z^i)^{\Tr}\delta \s^i
\nonumber \\
& = & \left( \delta \x^{i^{\Tr}} \H^{\Tr} + \delta \y^{i^{\Tr}}\A_E \right)
\delta \x^i-(\delta \z^i)^{\Tr}\delta \s^i
\nonumber \\
& = & \delta \x^{i^{\Tr}} \H^{\Tr} \delta \x^i - (\delta \z^i)^{\Tr}\delta \s^i=0.
\label{tmp3}
\end{eqnarray}
Since $\H$ is a convex Hessian matrix, therefore, $\H$ is symmetric and positive 
definite, the last equality indicates that
$(\delta \z^i)^{\Tr}\delta \s^i =\delta \x^{i^{\Tr}} \H \delta \x^i \ge 0$
for $i=1,2,3$.
\hfill \qed
\end{proof}

Now, we are ready to provide several estimations that are important to
the convergence analysis.

\begin{lemma}
Let $(\x^0, \y^0,\w^0,\s^0, \z^0)$ be the initial point of Algorithm \ref{mainAlgo1},
and $(\bar{\x},\bar{\y},\bar{\z},\bar{\w},\bar{\s})$ be an optimal
solution of (\ref{CP}). Then
\begin{equation}
\| \D  \dot{\z} \|, \| \D^{-1}  \dot{\s} \|
\le \sqrt{n\mu} + \| \D \delta \z^2  \| + \| \D^{-1} \delta \s^3 \|.
\label{mainIneq}
\end{equation}
\label{wright}
\end{lemma}
\begin{proof}
Since $(\D^{-1} \delta \s^i)^{\Tr} (\D \delta \z^i) \ge 0$, 
for $i=1,2,3$, it follows that
\begin{equation}
\| \D^{-1} \delta \s^i \|^2, \| \D \delta \z^i \|^2 
\le \| \D^{-1} \delta \s^i \|^2 + \| \D \delta \z^i \|^2
\le \| \D^{-1} \delta \s^i + \D \delta \z^i \|^2.
\label{general}
\end{equation}
Applying $\S \delta \z^i + \Z \delta \s^i= \r^i$ to (\ref{general}) 
for $i=1,2,3$ respectively, we obtain the following relations
\begin{subequations}
\begin{gather}
\| \D^{-1} \delta \s^1 \|, \| \D \delta \z^1 \| 
\le  \| \D^{-1} \delta \s^1 + \D \delta \z^1 \|
=\| (\S\z)^{\frac{1}{2}} \| = \sqrt{\s^{\Tr}\z} = \sqrt{n\mu}, \label{deltaX1S1} \\
\| \D^{-1} \delta \s^2 \|, \| \D \delta \z^2 \| 
\le  \| \D^{-1} \delta \s^2 + \D \delta \z^2 \|
=\nu_k \| \D^{-1} (\s^0-\bar{\s}) \| ,  \label{deltaX2S2} \\
\| \D^{-1} \delta \s^3 \|, \| \D \delta \z^3 \| 
\le  \| \D^{-1} \delta \s^3 + \D \delta \z^3 \|
=\nu_k \| \D (\z^0-\bar{\z}) \|. \label{deltaX3S3} 
\end{gather}
\label{deltaXS} 
\end{subequations} 
Considering the last row of (\ref{matrixSplit}) with $i=2$, we have
\[
\Z \delta \s^2 + \S \delta \z^2 = \r^2 = -\nu_k \Z(\s^0 - \bar{\s}),
\]
which is equivalent to 
\begin{equation}
\delta \s^2 = -\nu_k (\s^0 - \bar{\s}) - \D^2 \delta \z^2.
\label{deltax2}
\end{equation}
Thus, from (\ref{deltaS}), (\ref{deltax2}), and (\ref{deltaXS}), we have
\begin{eqnarray}
\| \D^{-1} \dot{\s} \| & = & 
\| \D^{-1} [\delta \s^1 +\delta \s^2 +\delta \s^3 + \nu_k (\s^0 - \bar{\s})] \|
\nonumber \\
& = & \| \D^{-1} \delta \s^1 - \D \delta \z^2 + \D^{-1} \delta \s^3 \|
\nonumber \\
& \le & \| \D^{-1} \delta \s^1 \| + \| \D \delta \z^2  \| + \| \D^{-1} \delta \s^3 \|.
\label{formuA}
\end{eqnarray}
Considering the last row of (\ref{matrixSplit}) with $i=3$, we have
\[
\Z\delta \s^3 + \S\delta \z^3 = \r^3 = -\nu_k \S(\z^0 - \bar{\z}),
\]
which is equivalent to 
\begin{equation}
\delta \z^3 = -\nu_k (\z^0 - \bar{\z}) - \D^{-2} \delta \s^3.
\label{deltas3}
\end{equation}
Thus, from (\ref{deltaZ}), (\ref{deltas3}), and (\ref{deltaXS}), we have
\begin{eqnarray}
\| \D  \dot{\z} \| & = & 
\| \D  [\delta \z^1 +\delta \z^2 +\delta \z^3 + \nu_k (\z^0 - \bar{\z})] \|
\nonumber \\
& = & \| \D  \delta \z^1 + \D \delta \z^2 - \D^{-1} \delta \s^3 \|
\nonumber \\
& \le & \| \D  \delta \z^1 \| + \| \D \delta \z^2  \| + \| \D^{-1} \delta \s^3 \|.\label{formuB}
\end{eqnarray}
In view of (\ref{deltaX1S1}), adjoining (\ref{formuA}) and
(\ref{formuB}) gives (\ref{mainIneq}).
\hfill \qed
\end{proof}

The next lemma provides an estimation which will be used
to establish the polynomiality for the proposed algorithm.

\begin{lemma}
Let $(\dot{\x}, \dot{\y}, \dot{\w}, \dot{\s}, \dot{\z})$ be 
defined in (\ref{firstOrder}). Then, there is a constant $C_2$
independent of $n$ and $m$ such that in every iteration of 
Algorithm \ref{mainAlgo1}, the following inequality holds.
\begin{equation}
\| \D  \dot{\z} \|, \| \D^{-1}  \dot{\s} \|
\le C_2 \sqrt{n \mu_k}.
\label{converg2}
\end{equation} 
\label{main1}
\end{lemma}
\begin{proof}
In view of Lemma \ref{inEq1}, we have 
\[ (\D\delta \z^2)^{\T} (\D^{-1} \delta \s^2)
=(\delta \z^2)^{\T} (\delta \s^2) \ge 0,
~~~~~~ (\D\delta \z^3)^{\T} (\D^{-1} \delta \s^3)
=(\delta \z^3)^{\T} (\delta \s^3) \ge 0. \]
Using a similar idea of \cite{yang18}, we can derive (\ref{converg2}).
\hfill \qed
\end{proof}

Lemma \ref{main1} can be used to derive several useful inequalities.

\begin{lemma}
Let $(\dot{\x}, \dot{\y}, \dot{\w}, \dot{\s}, \dot{\z})$ and
$(\ddot{\x}, \ddot{\y}, \ddot{\w}, \ddot{\s}, \ddot{\z})$ be 
defined in(\ref{firstOrder}) and (\ref{secondOrder}). There is a constant
$C_3>0$ independent of $n$ and $m$ such that the following 
inequalities hold.
\begin{subequations}
\begin{gather}
\| \D^{-1} \ddot{\s} \|, \| \D  \ddot{\z} \| \le C_3 n  \mu_k^{0.5}, 
\\
\| \D^{-1} \p_{\s} \|, \| \D  \p_{\z} \| \le \sqrt{\frac{n}{\theta}} \mu_k^{0.5}, 
\\
\| \D^{-1} \q_{\s} \|, \| \D  \q_{\z} \| \le 
\frac{2C_2^2 }{\sqrt{\theta}} n \mu_k^{0.5}.
\label{exten}
\end{gather}
\end{subequations}
\label{extension}
\end{lemma}
\begin{proof}
In view of the last row of (\ref{secondOrderQ}), 
using the facts that $\q_{\s}^{\T} \q_{\z} \ge 0$ (\ref{ineqC}), 
$s_i^k z_i^k > \theta \mu_k$, and Lemma
\ref{main1}, and a similar idea of \cite{yang18}, we can prove (\ref{exten}).
\hfill \qed
\end{proof}

The following inequalities follow directly from the results of 
Lemmas \ref{main1} and \ref{extension}.

\begin{lemma}
Let $(\dot{\x}, \dot{\y}, \dot{\w}, \dot{\s}, \dot{\z})$ and
$(\ddot{\x}, \ddot{\y}, \ddot{\w}, \ddot{\s}, \ddot{\z})$ be 
defined in(\ref{firstOrder}) and (\ref{secondOrder}).
The following inequalities hold.
\begin{equation}
\frac{| \dot{\s}^{\T} \dot{\z} |}{n} \le C_2^2 \mu_k, \hspace{0.1in}
\frac{| \ddot{\s}^{\T} \dot{\z}|}{n} \le C_2C_3 \sqrt{n} \mu_k, \hspace{0.1in}
\frac{| \dot{\s}^{\T} \ddot{\z}|}{n} \le C_2C_3 \sqrt{n} \mu_k. 
\label{rela1}
\end{equation}
Moreover,
\begin{equation}
| \dot{s}_i \dot{z}_i | \le C_2^2 n \mu_k, \hspace{0.1in}
| \ddot{s}_i \dot{z}_i| \le C_2C_3 n^{\frac{3}{2}} \mu_k, \hspace{0.1in}
| \dot{s}_i \ddot{z}_i| \le C_2C_3 n^{\frac{3}{2}} \mu_k, \hspace{0.1in}
| \ddot{s}_i \ddot{z}_i| \le C_3^2 n^2 \mu_k. 
\label{rela2}
\end{equation}
\label{inequalities}
\label{lemma5}
\end{lemma}

Lemmas \ref{main1}, \ref{extension}, and \ref{inequalities} will
be used in finding the lower bound of $\alpha_k$.

\begin{lemma} There is a positive constant $C_4$ independent of 
$m$ and $n$, and an $\bar{\alpha}$ defined by 
$\sin(\bar{\alpha}) \ge \frac{C_4}{\sqrt{n}}$ such that
for $\forall k \ge 0$ and $\sin(\alpha_k) \in (0,\sin(\bar{\alpha})]$,
\begin{equation}
(s_i^{k+1},z_i^{k+1}):=(s_i(\sigma_k,\alpha_k), z_i(\sigma_k,\alpha_k) ) 
\ge (\phi_k, \psi_k) >0
\label{posity}
\end{equation}
holds.
\label{positiveCond}
\end{lemma}
\begin{proof}
Using  (\ref{posi1}), (\ref{phi}), and Lemmas \ref{main1}, 
\ref{inequalities}, \ref{sincos}, the proof is straightforward 
and similar to the proof used in \cite{yang18}.
\hfill \qed
\end{proof}

\begin{lemma} 
There is a positive constant $C_5$ independent of $n$ and $m$,
and an $\hat{\alpha}$ defined by 
$\sin(\hat{\alpha}) \ge \frac{C_5}{n^{\frac{1}{4}}}$ such that
for $\forall k \ge 0$ and $\sin(\alpha) \in (0,\sin(\hat{\alpha})]$, the following relation
\begin{equation}
\mu_k (\sigma_k,\alpha_k) \le \mu_k \left( 1 - \frac{\sin(\alpha_k)}{4} \right)
\le \mu_k \left( 1 - \frac{C_5}{4n^{\frac{1}{4}}} \right)
\label{polynoty}
\end{equation}
\label{objDec}
\end{lemma}
holds.
\begin{proof}
Using Proposition \ref{simple2}, Lemmas \ref{sincos} and \ref{inequalities}
and the similar idea used in \cite{yang18}, 
we can easily prove the result. 
\hfill\qed
\end{proof}

\begin{lemma} 
There is a positive constant $C_6$ independent of $n$ and $m$, an
$\check{\alpha}$ defined by $\sin(\check{\alpha}) \ge \frac{C_6}{n^{\frac{3}{2}}}$ 
such that if $s_i^{k}z_i^{k} \ge \theta \mu_{k}$ holds, then for $\forall k \ge 0$,
$\forall i \in \{ 1, \ldots, n\}$, and $\sin(\alpha) \in (0,\sin(\check{\alpha})]$, the following relation
\begin{equation}
s_i^{k+1}z_i^{k+1} \ge \theta \mu_{k+1}
\label{polycompl}
\end{equation}
\label{barrier}
\end{lemma}
holds.
\begin{proof}
Using (\ref{thirdBound1}), Proposition \ref{simple2}, and Lemma \ref{lemma5}
and the similar idea used in \cite{yang18}, 
we can easily prove the result. 
\hfill\qed
\end{proof}

The following theorem given in \cite{wright97} has been used to
establish the polynomial bound for almost all interior-point algorithms.

\begin{theorem}[\cite{wright97}]
Let $\epsilon \in (0,1)$ be given. Suppose that an algorithm generates a sequence of iterations 
$\{ \chi_k \}$ that satisfies
\begin{equation}
\chi_{k+1} \le \left( 1 - \frac{\delta}{n^{\omega}} \right) \chi_k, 
\hspace{0.1in} k=0, 1, 2, \ldots,
\end{equation}
for some positive constants $\delta$ and $\omega$. Then there 
exists an index $K$ with
\[
K={\mathcal O}(n^{\omega}\log({\chi_0}/{\epsilon}))
\]
such that 
\[
\chi_k \le \epsilon \hspace{0.1in} {\rm for} \hspace{0.1in} \forall k \ge K.
\]
\hfill \qed
\label{white97}
\end{theorem}

The main result of the paper immediately follows from Lemmas \ref{basic}, 
\ref{positiveCond}, \ref{objDec}, \ref{barrier}, and Theorem \ref{white97}.

\begin{theorem}
Algorithms~\ref{mainAlgo1} is a polynomial algorithm with
polynomial complexity bound of 
${\mathcal O}({n}^{\frac{3}{2}} \max \{ \log({(\s^0)^{\T}\z^0}/{\epsilon}), 
\log({\r_C^0}/{\epsilon}),\log({\r_B^0}/{\epsilon}),\log({\r_I^0}/{\epsilon}) \} )$.
\end{theorem}
\begin{proof}
The proof is similar to the one used in \cite{yang18} and therefore omitted.
\end{proof}

\section{Implementation and numerical test}\label{sec:test}

In this section, we briefly discuss a Matlab implementation of the
proposed algorithm and provide some preliminary test results.

\subsection{Matlab implementation}

Algorithm \ref{mainAlgo1} is  implemented as a Matlab function:
\newline \centerline
{\tt [x,obj,kk,infe]=arcConvex(AE,bE,AI,bI,f,v0,d)}
\newline
where AE and bE are the input matrix and vector for equality constraints
\footnote{If there is no equality constraint in the problem, then the
inputs of AE and bE are [ ] and [ ] respectively.}, 
AI and bI are the input matrix and vector for inequality constraints,
$f$ is the objective function of the convex nonlinear optimization problem,
which calls a function handle created in a separate file; $\x$ is the output
which returns the optimal solution, {\it obj} is the output which returns the 
optimal value of the convex nonlinear optimization problem, {\it kk} 
is the iteration number which is used to find the optimal solution, 
and {\it infe} is the norm of $\| \A_E \x - \b_E \|$ which should be small
when the program terminates.

In Algorithm \ref{mainAlgo1}, Step 1 involves the calculation of the
gradient of $\g(\v^k)$, and Step 2 involves the calculation of the 
gradient of $\nabla_{\x} L(\v^k)$ and the Hessian 
$\H=\nabla_{\x}^2 L(\v^k)$. To avoid manipulating analytical formulas
for every individual problem, which can be tedious and  
error prone, we adopted a piece of code used in \cite{yang15} that 
implements the automatic differentiation method discussed in \cite{nw06}. 
The optimal section of $\sigma$ and $\alpha$ is implemented exactly
as described in the Algorithms \ref{bisectionSearch} and \ref{Step5}.

\subsection{Preliminary numerical test}

The implemented Matlab code is tested for a few problems
that were found from different sources. We made no effort to
select the initial points for these problems. All problems have
a similar linear constraint set but different objective functions.

\vspace{0.03in}
\noindent
{\bf Example 1} \cite{kumar21}: This problem was posted in Researchgate
and a solution was solicited. The objective function is a two dimensional
logarithm function.
\begin{eqnarray} \nonumber 
& \min & -[(a_1 \log (x_1) - x_1 + b_1) + (a_2  \log (x_2) - x_2 + b_2)]
\\ \nonumber 
& s.t. & x_1+x_2 \le 10  
\\ \label{prob1}
& & \ell_1 \le x_1 \le u_1
\\ \nonumber 
& &  \ell_2 \le x_2 \le u_2,
\end{eqnarray}
where we set $a_1=5$, $b_1=7$, $a_2=7$, $b_2=8$, $\ell_1=1$, 
$\ell_2=1$, $u_1=10$, and $u_2=10$. Starting from initial point 
$\x^0=(5, 5)$, $\w^0=\z^0= (100, 100, 100, 100, 100)$, and
$\s^0= (0.01, 0.01, 0.01, 0.01, 0.01)$, after $68$ iterations,
we get $\x=(1,1)$ and the optimal value is $-13$.

\vspace{0.03in}
The rest examples are created based on \cite[pages 71-73]{bv14}.

\vspace{0.03in}
\noindent
{\bf Example 2} The objective function of this problem is a 
two dimensional exponential function.
\begin{eqnarray} \nonumber 
& \min & (a_1 e^{x_1} + b_1) + (a_2  e^{x_2} + b_2)
\\ \nonumber 
& s.t. & x_1+x_2 \le 10  
\\ \label{prob2}
& & \ell_1 \le x_1 \le u_1
\\ \nonumber 
& &  \ell_2 \le x_2 \le u_2,
\end{eqnarray}
where we set $a_1=5$, $b_1=7$, $a_2=7$, $b_2=8$, $\ell_1=2$, 
$\ell_2=1$, $u_1=10$, and $u_2=10$. Starting from initial point 
$\x^0=(5, 5)$, $\w^0=\z^0= (100, 100, 100, 100, 100)$, and
$\s^0= (0.01, 0.01, 0.01, 0.01, 0.01)$, after $66$ iterations,
we get $\x=(2,1)$ and the optimal value is $70.9733$.

\vspace{0.03in}
\noindent
{\bf Example 3} The objective function of this problem is a 
two dimensional power function of $\x^{\a}$ with the first power
greater than $1$ and the second power smaller than $0$.
\begin{eqnarray} \nonumber 
& \min & (a_1 x_1^3 + b_1) + (a_2  \frac{1}{x_2} + b_2)]
\\ \nonumber 
& s.t. & x_1+x_2 \le 10  
\\ \label{prob3}
& & \ell_1 \le x_1 \le u_1
\\ \nonumber 
& &  \ell_2 \le x_2 \le u_2,
\end{eqnarray}
where we set $a_1=5$, $b_1=7$, $a_2=7$, $b_2=8$, $\ell_1=1$, 
$\ell_2=2$, $u_1=10$, and $u_2=10$. Starting from initial point 
$\x^0=(5, 5)$, $\w^0=\z^0= (100, 100, 100, 100, 100)$, and
$\s^0= (0.01, 0.01, 0.01, 0.01, 0.01)$, after $69$ iterations,
we get $\x=(1,2)$ and the optimal value is $23.5000$.

\vspace{0.03in}
\noindent
{\bf Example 4} The objective function of this problem is a 
two dimensional negative entropy function.
\begin{eqnarray} \nonumber 
& \min & (a_1 x_1\log (x_1) + b_1) + (a_2  x_2 \log (x_2) + b_2)
\\ \nonumber 
& s.t. & x_1+x_2 \le 10  
\\ \label{prob4}
& & \ell_1 \le x_1 \le u_1
\\ \nonumber 
& &  \ell_2 \le x_2 \le u_2,
\end{eqnarray}
where we set $a_1=5$, $b_1=7$, $a_2=7$, $b_2=8$, $\ell_1=2$, 
$\ell_2=2$, $u_1=10$, and $u_2=10$. Starting from initial point 
$\x^0=(5, 5)$, $\w^0=\z^0= (100, 100, 100, 100, 100)$, and
$\s^0= (0.01, 0.01, 0.01, 0.01, 0.01)$, after $69$ iterations,
we get $\x=(2,2)$ and the optimal value is $31.6355$.

\vspace{0.03in}
\noindent
{\bf Example 5} The objective function of this problem is a 
two dimensional quadratic-over-linear function.
\begin{eqnarray} \nonumber 
& \min & \frac{(a_1 x_1)^2}{a_2 x_2} 
\\ \nonumber 
& s.t. & x_1+x_2 \le 10  
\\ \label{prob5}
& & \ell_1 \le x_1 \le u_1
\\ \nonumber 
& &  \ell_2 \le x_2 \le u_2,
\end{eqnarray}
where we set $a_1=5$, $a_2=7$,  $\ell_1=1$, 
$\ell_2=3$, $u_1=10$, and $u_2=10$. Starting from initial point 
$\x^0=(5, 5)$, $\w^0=\z^0= (100, 100, 100, 100, 100)$, and
$\s^0= (0.01, 0.01, 0.01, 0.01, 0.01)$, after $57$ iterations,
we get $\x=(4.9271,5.0595)$ and the optimal value is $17.1360$.

\vspace{0.03in}
\noindent
{\bf Example 6} The objective function of this problem
is a two dimensional log-sum-exponential function.
\begin{eqnarray} \nonumber 
& \min & \log(a_1 e^{x_1}+ a_2 e^{x_2})
\\ \nonumber 
& s.t. & x_1+x_2 \le 10  
\\ \label{prob6}
& & \ell_1 \le x_1 \le u_1
\\ \nonumber 
& &  \ell_2 \le x_2 \le u_2,
\end{eqnarray}
where we set $a_1=5$, $a_2=7$, $\ell_1=3$, 
$\ell_2=1$, $u_1=10$, and $u_2=10$. Starting from initial point 
$\x^0=(5, 5)$, $\w^0=\z^0= (100, 100, 100, 100, 100)$, and
$\s^0= (0.01, 0.01, 0.01, 0.01, 0.01)$, after $56$ iterations,
we get $\x=(4.9924,4.9924)$ and the optimal value is $7.4773$.

\vspace{0.03in}
\noindent
{\bf Example 7} The objective function of this problem is 
a two dimensional geometric mean function.
\begin{eqnarray} \nonumber 
& \min & ( x_1 * x_2)^{1/2}
\\ \nonumber 
& s.t. & x_1+x_2 \le 10  
\\ \label{prob7}
& & \ell_1 \le x_1 \le u_1
\\ \nonumber 
& &  \ell_2 \le x_2 \le u_2,
\end{eqnarray}
where we set  $\ell_1=2$, 
$\ell_2=3$, $u_1=10$, and $u_2=10$. Starting from initial point 
$\x^0=(5, 5)$, $\w^0=\z^0= (100, 100, 100, 100, 100)$, and
$\s^0= (0.01, 0.01, 0.01, 0.01, 0.01)$, after $59$ iterations,
we get $\x=(2.0006,7.9767)$ and the optimal value is $3.9948$.

\vspace{0.03in}
\noindent
{\bf Example 8} The objective function of this problem is a two 
dimensional log-determinant function for positive definite matrix.
\begin{eqnarray} \nonumber 
& \min & -\log \det \left( \left[ \begin{array} {cc} 
x_1 &  x_2 \\ x_2 & x_3 \end{array} \right] \right)
\\ \nonumber 
& s.t. & x_1+x_2 \le 10  
\\ \nonumber 
&   & x_2+x_3 \le 10 
\\ \label{prob8}
& & \ell_1 \le x_1 \le u_1
\\ \nonumber 
& &  \ell_2 \le x_2 \le u_2,
\\ \nonumber 
& &  \ell_3 \le x_3 \le u_3,
\end{eqnarray}
where we set  $\ell_1=5$, $\ell_2=1$, $\ell_3=5$, $u_1=10$, 
$u_2=3$, and $u_3=10$. Starting from initial point $\x^0=(6, 2, 6)$,
$\w^0=\z^0= (100, 100, 100, 100, 100, 100, 100, 100)$, and
$\s^0= (0.01, 0.01, 0.01, 0.01, 0.01, 0.01, 0.01, 0.01)$, after $44$ iterations,
we get $\x=(5,3,5)$ and the optimal value is $-2.7726$.

\section{Conclusions}\label{sec:conclusion}

In this paper, we proposed an infeasible interior-point arc-search algorithm
for convex optimization problem with linear equality and inequality constraints.
Many application problems can be formulated as this optimization problem. 
We showed that this algorithm is convergent with a nice polynomial iteration
bound. To have a good performance, we provided analytic formulas for the
the arc-search, and developed an efficient algorithm to dynamically select
centering parameter and the step size at the same time in 
Appendix A. In the future, we may consider the general convex 
programming problem with convex monlinear inequality constraints.

\section{acknowledgements}
This author thanks the anonymous referees for their very detailed and 
constructive comments. The quality of the paper has been significantly 
improved through addressing these thoughtful comments.

%
%

\appendix

\section{Selection of the centering parameter $\sigma_k$ and step size $\alpha_k$}
\label{sec:selection}

Although the method of selecting $\alpha_k$ described in Sections 
\ref{sec:algorithm} and \ref{sec:convergence} assures that the 
algorithm converges in polynomial iteration, but this selection is very 
conservative. A better method is to simultaneously select centering
parameter $\sigma_k$ and step size $\alpha_k$ to maximize
the step size in every iteration.
The merit of this holistic strategy is proved in theory \cite{yang21}, 
and has been demonstrated in computational experiments 
\cite{yang17,yang18}. The same strategy is proposed in Step 5
of Algorithm \ref{mainAlgo1}, but there is no details provided there. 
In this appendix, we discuss how this strategy is implemented.
Although the formulas in this appendix are similar to the ones in
\cite{yang18}, they are different. To avoid the confusion and
implementation errors, we would like to list them in this appendix.
For the sake of completeness, we also provide the proofs even
though they follow the same ideas of \cite{yang18}.

Let the current iterate be $\v^k=(\x^k,\y^k,\w^k,\s^k,\z^k)$, 
$(\dot{\x},\dot{\y},\dot{\w},\dot{\s},\dot{\z})$ be computed 
by solving (\ref{firstOrder}), 
$(\p_{\x},\p_{\y},\p_{\w},\p_{\s},\p_{\z})$ be computed
by solving (\ref{secondOrderP}) and 
$(\q_{\x},\q_{\y},\q_{\w},\q_{\s},\q_{\z})$
be computed by solving (\ref{secondOrderQ}), $\phi_k$ and $\psi_k$
be computed by using (\ref{ijmath}) and (\ref{phikpsi}).
An intuition based on Propositions \ref{basic} and \ref{simple2} is that
the step size $\alpha_k$ should be chosen as large as possible 
provided that Condition (C4), (\ref{posi1}) and (\ref{posi2}) hold.
Given $\v^k=(\x^k,\y^k,\w^k,\s^k,\z^k)$, 
$(\dot{\x},\dot{\y},\dot{\w},\dot{\s},\dot{\z})$,
$(\p_{\x},\p_{\y},\p_{\w},\p_{\s},\p_{\z})$, 
$(\q_{\x},\q_{\y},\q_{\w},\q_{\s},\q_{\z})$,
$\phi_k$ and $\psi_k$, similar to the derivation of \cite{yang17},
the largest $\tilde{\alpha}$ that meet conditions (\ref{posi1}) and (\ref{posi2})
can be expressed as a function of $\sigma_k$. For each 
$i \in \lbrace 1,\ldots, n \rbrace$, given $\sigma$, we can select the largest
$\alpha_{s_i}$ such that for any $\alpha \in [0, \alpha_{s_i}]$, 
the $i$th inequality of (\ref{posi1}) holds, and the largest 
$\alpha_{z_i}$ such that for any $\alpha \in [0, \alpha_{z_i}]$ 
the $i$th inequality of (\ref{posi2}) holds. We then define 
\begin{eqnarray}
{\alpha^s}=\displaystyle\min_{i \in \lbrace 1,\ldots, n \rbrace}
\lbrace \alpha_{s_i}\rbrace, \label{alphaS}
\\
{\alpha^z}=\displaystyle\min_{i \in \lbrace 1,\ldots, n \rbrace}
\lbrace \alpha_{z_i} \rbrace,  \label{alphaZ}
\\
{\tilde{\alpha}}=\min \lbrace \alpha^s, \alpha^z \rbrace,
\label{alpha}
\end{eqnarray}
where $\alpha_{s_i}$ and $\alpha_{z_i}$ can be obtained, 
using a similar argument as in \cite{yang17}, in 
analytical forms represented by $\phi_k$, $\dot{s}_i$, 
$\ddot{s}_i=p_{s_i}\sigma+q_{s_i}$, $\psi_k$, $\dot{z}_i$, 
and $\ddot{z}_i=p_{z_i}\sigma+q_{z_i}$. 
First, from (\ref{posi1}), we have
\begin{equation}
s_i +\ddot{s}_i - \phi_k
\ge   \dot{s}_i\sin(\alpha)+\ddot{s}_i\cos(\alpha).
\label{alphai}
\end{equation}

\vspace{0.1in}

\noindent{\it Case 1a ($\dot{s}_i=0$ and $p_{s_i}\sigma+q_{s_i} \ne 0$)}:

In this case, if $\ddot{s}_i \ge -(s_i-\phi_k)$ and  
$\alpha \in [0, \frac{\pi}{2}]$, then $s_i(\alpha) \ge \phi_k$ follows from 
(\ref{posi1}). If $\ddot{s}_i \le -(s_i-\phi_k)$ or 
$s_i +\ddot{s}_i - \phi_k \le 0$, to meet (\ref{alphai}),
we must have 
$\cos(\alpha) \ge \frac{x_i +\ddot{x}_i-\phi_k}{\ddot{x}_i}$, or,
$\alpha \le \cos^{-1}\left( \frac{x_i +\ddot{x}_i-\phi_k}
{\ddot{x}_i} \right)$. Therefore,
\begin{equation}
\alpha_{s_i}(\sigma) = \left\{
\begin{array}{ll}
\frac{\pi}{2} & \quad \mbox{if $s_i +(p_{s_i}\sigma+q_{s_i})-\phi_k \ge 0 $} \\
\cos^{-1}\left( \frac{s_i -\phi_k +p_{s_i}\sigma+q_{s_i}}
{p_{s_i}\sigma+q_{s_i}} \right) & \quad 
\mbox{if $s_i +(p_{s_i}\sigma+q_{s_i})-\phi_k \le 0 $}.
\end{array}
\right.
\label{case1a}
\end{equation}
\noindent{\it Case 2a ($p_{s_i}\sigma+q_{s_i}=0$ and $\dot{s}_i\ne 0$)}:

In this case, if $\dot{s}_i \le s_i -\phi_k$ and $\alpha \in [0, \frac{\pi}{2}]$, 
then $s_i(\alpha) \ge  \phi_k$ follows from (\ref{posi1}).
If $\dot{s}_i \ge s_i -\phi_k$,  to meet (\ref{alphai}),
we must have 
$\sin(\alpha) \le \frac{s_i  -\phi_k}{\dot{s}_i}$, or
$\alpha \le \sin^{-1}\left( \frac{s_i  -\phi_k}
{\dot{s}_i} \right)$. Therefore,
\begin{equation}
\alpha_{s_i}(\sigma)  = \left\{
\begin{array}{ll}
\frac{\pi}{2} & \quad \mbox{if $\dot{s}_i \le s_i -\phi_k $} \\
\sin^{-1}\left( \frac{s_i-\phi_k }
{\dot{s}_i} \right) & \quad \mbox{if $\dot{s}_i \ge s_i -\phi_k $}
\end{array}
\right.
\label{case2a}
\end{equation}
\noindent{\it Case 3a ($\dot{s}_i>0$ and $p_{s_i}\sigma+q_{s_i}>0$)}:

Let $\dot{s}_i=\sqrt{\dot{s}_i^2+\ddot{s}_i^2}\cos(\beta)$, and
$\ddot{s}_i=\sqrt{\dot{s}_i^2+\ddot{s}_i^2}\sin(\beta)$, (\ref{alphai})
can be rewritten as 
\begin{equation}
s_i + \ddot{s}_i  -\phi_k \ge \sqrt{\dot{s}_i^2+\ddot{s}_i^2}
\sin(\alpha + \beta),
\label{alphai1}
\end{equation}
where 
\begin{equation}
\beta = \sin^{-1} \left( \frac{\ddot{s}_i }
{\sqrt{\dot{s}_i^2+\ddot{s}_i^2}} \right)
= \sin^{-1}\left( \frac{p_{s_i}\sigma+q_{s_i} } 
{\sqrt{\dot{s}_i^2+(p_{s_i}\sigma+q_{s_i})^2}} \right).
\label{beta1}
\end{equation}
If $\ddot{s}_i + s_i  -\phi_k\ge \sqrt{\dot{s}_i^2+\ddot{s}_i^2}$ 
and $\alpha \in [0, \frac{\pi}{2}]$, then $s_i(\alpha) \ge  \phi_k$ 
follows from (\ref{posi1}).
If $\ddot{s}_i + s_i -\phi_k \le \sqrt{\dot{s}_i^2+\ddot{s}_i^2}$,  
to meet (\ref{alphai1}), we must have 
$\sin(\alpha + \beta) \le \frac{s_i + \ddot{s}_i -\phi_k}
{\sqrt{\dot{s}_i^2+\ddot{s}_i^2}}$, or
$\alpha + \beta \le \sin^{-1}\left( \frac{s_i + \ddot{s}_i  -\phi_k}
{\sqrt{\dot{s}_i^2+\ddot{s}_i^2}} \right)$. Therefore,
\begin{equation}
\alpha_{s_i}(\sigma)  = \left\{
\begin{array}{ll}
\frac{\pi}{2} & \quad \mbox{if $s_i-\phi_k + p_{s_i}\sigma+q_{s_i} \ge 
\sqrt{\dot{s}_i^2+(p_{s_i}\sigma+q_{s_i})^2}$} \\
\sin^{-1}\left( \frac{s_i -\phi_k + p_{s_i}\sigma+q_{s_i} }
{\sqrt{\dot{s}_i^2+(p_{s_i}\sigma+q_{s_i})^2}} \right) 
- \beta & \quad 
\mbox{if $s_i -\phi_k + p_{s_i}\sigma+q_{s_i} \le 
\sqrt{\dot{s}_i^2+(p_{s_i}\sigma+q_{s_i})^2}$}
\end{array}
\right.
\label{case3a}
\end{equation}
\noindent{\it Case 4a ($\dot{s}_i>0$ and $p_{s_i}\sigma+q_{s_i}<0$)}:

Let $\dot{s}_i=\sqrt{\dot{s}_i^2+\ddot{s}_i^2}\cos(\beta)$, and
$\ddot{s}_i=-\sqrt{\dot{s}_i^2+\ddot{s}_i^2}\sin(\beta)$, (\ref{alphai})
can be rewritten as 
\begin{equation}
s_i + \ddot{s}_i -\phi_k  \ge \sqrt{\dot{s}_i^2+\ddot{s}_i^2}
\sin(\alpha - \beta),
\label{alphai2}
\end{equation}
where 
\begin{equation}
\beta = \sin^{-1} \left( \frac{-\ddot{s}_i }
{\sqrt{\dot{s}_i^2+\ddot{s}_i^2}} \right)
= \sin^{-1}\left( \frac{-(p_{s_i}\sigma+q_{s_i} )}
{\sqrt{\dot{s}_i^2+(p_{s_i}\sigma+q_{s_i})^2}} \right).
\label{beta2}
\end{equation}
If $\ddot{s}_i + s_i -\phi_k \ge \sqrt{\dot{s}_i^2+\ddot{s}_i^2}$ and 
$\alpha \in [0, \frac{\pi}{2}]$, then $s_i(\alpha) \ge \phi_k$ 
follows from (\ref{posi1}).
If $\ddot{s}_i + s_i -\phi_k \le \sqrt{\dot{s}_i^2+\ddot{s}_i^2}$,  
to meet (\ref{alphai2}), we must have 
$\sin(\alpha - \beta) \le \frac{s_i + \ddot{s}_i}
{\sqrt{\dot{s}_i^2+\ddot{s}_i^2}}$, or
$\alpha - \beta \le \sin^{-1}\left( \frac{s_i + \ddot{s}_i }
{\sqrt{\dot{s}_i^2+\ddot{s}_i^2}} \right)$. Therefore,
\begin{equation}
\alpha_{s_i}(\sigma)  = \left\{
\begin{array}{ll}
\frac{\pi}{2} & \quad \mbox{if $s_i -\phi_k + p_{s_i}\sigma+q_{s_i} \ge 
\sqrt{\dot{s}_i^2+(p_{s_i}\sigma+q_{s_i})^2}$} \\
\sin^{-1}\left( \frac{s_i-\phi_k + p_{s_i}\sigma+q_{s_i} }
{\sqrt{\dot{s}_i^2+(p_{s_i}\sigma+q_{s_i})^2}} \right) 
+ \beta & \quad 
\mbox{if $s_i-\phi_k + p_{s_i}\sigma+q_{s_i} \le 
\sqrt{\dot{s}_i^2+(p_{s_i}\sigma+q_{s_i})^2}$}
\end{array}
\right.
\label{case4a}
\end{equation}
\noindent{\it Case 5a ($\dot{s}_i<0$ and $p_{s_i}\sigma+q_{s_i}<0$)}:

Let $\dot{s}_i=-\sqrt{\dot{s}_i^2+\ddot{s}_i^2}\cos(\beta)$, and
$\ddot{s}_i=-\sqrt{\dot{s}_i^2+\ddot{s}_i^2}\sin(\beta)$, (\ref{alphai})
can be rewritten as 
\begin{equation}
s_i + \ddot{s}_i  -\phi_k \ge -\sqrt{\dot{s}_i^2+\ddot{s}_i^2}
\sin(\alpha +\beta),
\label{alphai4}
\end{equation}
where 
\begin{equation}
\beta = \sin^{-1} \left( \frac{-\ddot{s}_i }
{\sqrt{\dot{s}_i^2+\ddot{s}_i^2}} \right)
= \sin^{-1}\left( \frac{-(p_{s_i}\sigma+q_{s_i}) }
{\sqrt{\dot{s}_i^2+(p_{s_i}\sigma+q_{s_i})^2}} \right).
\label{beta3}
\end{equation}
If $\ddot{s}_i + s_i -\phi_k \ge 0$ and 
$\alpha \in [0, \frac{\pi}{2}]$, then $s_i(\alpha) \ge  \phi_k $ 
follows from (\ref{posi1}). If $\ddot{s}_i + s_i -\phi_k \le 0$,  
to meet (\ref{alphai4}), we must have 
$\sin(\alpha + \beta) \ge \frac{-(s_i + \ddot{s}_i -\phi_k)}
{\sqrt{\dot{s}_i^2+\ddot{s}_i^2}}$, or
$\alpha + \beta \le \pi - \sin^{-1} \left( \frac{-(s_i + 
\ddot{s}_i -\phi_k)}{\sqrt{\dot{s}_i^2+\ddot{s}_i^2}} \right)$. 
Therefore,
\begin{equation}
\alpha_{s_i}(\sigma)  = \left\{
\begin{array}{ll}
\frac{\pi}{2} & \quad \mbox{if $s_i-\phi_k + p_{s_i}\sigma+q_{s_i}\ge 0$} \\
\pi - \sin^{-1} \left( \frac{-(s_i -\phi_k + p_{s_i}\sigma+q_{s_i}) }
{\sqrt{\dot{s}_i^2+(p_{s_i}\sigma+q_{s_i})^2}} \right) 
- \beta 
& \quad \mbox{if $s_i-\phi_k + p_{s_i}\sigma+q_{s_i} \le 0$}
\end{array}
\right.
\label{case5a}
\end{equation}
\noindent{\it Case 6a ($\dot{s}_i<0$ and $p_{s_i}\sigma+q_{s_i}>0$)}:

\begin{equation}
\alpha_{s_i}(\sigma)  = \frac{\pi}{2}.
\end{equation}
\noindent{\it Case 7a ($\dot{s}_i=0$ and $p_{s_i}\sigma+q_{s_i}=0$)}:

\begin{equation}
\alpha_{s_i}(\sigma)  = \frac{\pi}{2}.
\end{equation}

Using the same idea, we can obtain the similar formulas for 
$\alpha_{z_i}(\sigma)$.

\vspace{0.05in}
\noindent{\it Case 1b ($\dot{z}_i=0$, $p_{z_i}\sigma+q_{z_i} \ne 0$)}:

\begin{equation}
\alpha_{z_i}(\sigma)  = \left\{
\begin{array}{ll}
\frac{\pi}{2} & \quad \mbox{if $z_i-\psi_k +p_{z_i}\sigma+q_{z_i} \ge 0 $} \\
\cos^{-1}\left( \frac{z_i -\psi_k +p_{z_i}\sigma+q_{z_i}}
{p_{z_i}\sigma+q_{z_i}} \right) & \quad 
\mbox{if $z_i-\psi_k +p_{z_i}\sigma+q_{z_i} \le 0 $}.
\end{array}
\right.
\label{case1b}
\end{equation}
\noindent{\it Case 2b ($p_{z_i}\sigma+q_{z_i}=0$ and $\dot{z}_i \ne 0$)}:

\begin{equation}
\alpha_{z_i}(\sigma)  = \left\{
\begin{array}{ll}
\frac{\pi}{2} & \quad \mbox{if $\dot{z}_i \le z_i -\psi_k $} \\
\sin^{-1}\left( \frac{z_i -\psi_k }
{\dot{z}_i} \right) & \quad \mbox{if $\dot{z}_i \ge z_i -\psi_k $}
\end{array}
\right.
\label{case2b}
\end{equation}
\noindent{\it Case 3b ($\dot{z}_i>0$ and $p_{z_i}\sigma+q_{z_i}>0$)}:

Let 
\begin{equation}
\beta = \sin^{-1}\left( \frac{p_{z_i}\sigma+q_{z_i}} 
{\sqrt{\dot{z}_i^2+(p_{z_i}\sigma+q_{z_i})^2}} \right).
\label{beta1b}
\end{equation}
\begin{equation}
\alpha_{z_i}(\sigma)  = \left\{
\begin{array}{ll}
\frac{\pi}{2} & \quad \mbox{if $z_i -\psi_k + p_{z_i}\sigma+q_{z_i} \ge 
\sqrt{\dot{z}_i^2+(p_{z_i}\sigma+q_{z_i})^2}$} \\
\sin^{-1}\left( \frac{z_i -\psi_k + p_{z_i}\sigma+q_{z_i} }
{\sqrt{\dot{z}_i^2+(p_{z_i}\sigma+q_{z_i})^2}} \right)
- \beta & \quad 
\mbox{if $z_i -\psi_k + p_{z_i}\sigma+q_{z_i} < 
\sqrt{\dot{z}_i^2+(p_{z_i}\sigma+q_{z_i})^2}$}
\end{array}
\right.
\label{case3b}
\end{equation}
\noindent{\it Case 4b ($\dot{z}_i>0$ and $p_{z_i}\sigma+q_{z_i}<0$)}:

Let 
\begin{equation}
\beta = \sin^{-1}\left( \frac{-(p_{z_i}\sigma+q_{z_i})}
{\sqrt{\dot{z}_i^2+(p_{z_i}\sigma+q_{z_i})^2}} \right).
\label{beta2b}
\end{equation}
\begin{equation}
\alpha_{z_i}(\sigma)  = \left\{
\begin{array}{ll}
\frac{\pi}{2} & \quad \mbox{if $z_i -\psi_k + p_{z_i}\sigma+q_{z_i} \ge 
\sqrt{\dot{z}_i^2+(p_{z_i}\sigma+q_{z_i})^2}$} \\
\sin^{-1}\left( \frac{z_i -\psi_k + p_{z_i}\sigma+q_{z_i}}
{\sqrt{\dot{z}_i^2+(p_{z_i}\sigma+q_{z_i})^2}} \right) 
+ \beta & \quad 
\mbox{if $z_i-\psi_k + p_{z_i}\sigma+q_{z_i} \le 
\sqrt{\dot{z}_i^2+(p_{z_i}\sigma+q_{z_i})^2}$}
\end{array}
\right.
\label{case4b}
\end{equation}
\noindent{\it Case 5b ($\dot{z}_i<0$ and $p_{z_i}\sigma+q_{z_i}<0$)}:

Let 
\begin{equation}
\beta = \sin^{-1}\left( \frac{-(p_{z_i}\sigma+q_{z_i}) }
{\sqrt{\dot{z}_i^2+(p_{z_i}\sigma+q_{z_i})^2}} \right).
\label{beta3b}
\end{equation}
\begin{equation}
\alpha_{z_i}(\sigma)  = \left\{
\begin{array}{ll}
\frac{\pi}{2} & \quad \mbox{if $z_i-\psi_k + p_{z_i}\sigma+q_{z_i} \ge 0$} \\
\pi - \sin^{-1} \left( \frac{-(z_i -\psi_k + p_{z_i}\sigma+q_{z_i}) }
{\sqrt{\dot{z}_i^2+(p_{z_i}\sigma+q_{z_i})^2}} \right) 
- \beta & \quad \mbox{if $z_i -\psi_k + p_{z_i}\sigma+q_{z_i} \le 0$}
\end{array}
\right.
\label{case5b}
\end{equation}
\noindent{\it Case 6b ($\dot{z}_i<0$ and $p_{z_i}\sigma+q_{z_i}>0$)}:

\begin{equation}
\alpha_{z_i}(\sigma)  = \frac{\pi}{2}.
\end{equation}
\noindent{\it Case 7b ($\dot{z}_i=0$ and $p_{z_i}\sigma+q_{z_i}=0$)}:

\begin{equation}
\alpha_{z_i}(\sigma)  = \frac{\pi}{2}.
\label{case7b}
\end{equation}

Using this analytic formulas, our strategy to reduce the duality gap 
is to simultaneously select $\alpha_k$ and $\sigma_k$ by an 
iterative method similar to the idea of \cite{yang18}.
This is implemented as follows: in every iteration $k$, given fixed
$\phi_k$, $\psi_k$, $\dot{\s}$, $\dot{\z}$, $\p_{\s}$, $\p_{\z}$, 
$\q_{\s}$ and $\q_{\z}$, several different values of $\sigma$ are 
tried to find the best $\sigma_k$ for the maximum of $\tilde{\alpha}$.
Therefore, we will find a $\sigma_k$ which maximizes the step size
$\tilde{\alpha}$, i.e.,
\begin{equation}
\displaystyle\max_{\sigma \in [\sigma_{\min},\sigma_{\max}]} 
\hspace{0.05in}
\displaystyle\min_{i \in \{ 1, \ldots,n\} }
\{ \alpha_{s_i}(\sigma), \alpha_{z_i}(\sigma) \},
\label{maxmin}
\end{equation}
where $0 \le \sigma_{\min} < \sigma_{\max} \le 1$,
$\alpha_{s_i}(\sigma)$ and $\alpha_{z_i}(\sigma)$ are calculated 
using (\ref{case1a})-(\ref{case7b}) for 
$\sigma \in [\sigma_{\min},\sigma_{\max}]$. 
Problem (\ref{maxmin}) has no regularity conditions involving derivatives. 
Golden section search for variable $\sigma$ \cite{ekefer53} 
seems to be an appropriate method for solving this 
problem. Noting the fact from (\ref{posi1}) that 
$\alpha_{s_i}({\sigma})$ is a monotonic increasing function of 
$\sigma$ if $p_{s_i}>0$ and $\alpha_{s_i}({\sigma})$ is a 
monotonic decreasing function of $\sigma$ if $p_{s_i}<0$ 
(and similar properties hold for $\alpha_{z_i}(\sigma)$),
we can use the condition
\begin{equation}
\min \{ \displaystyle \min_{\{i \in p_{s_i}<0\}} \alpha_{s_i}(\sigma),
\displaystyle \min_{\{i \in p_{z_i}<0\}} \alpha_{z_i}(\sigma) \}
>
\min \{ \displaystyle \min_{\{i \in p_{s_i}>0\}} \alpha_{s_i}(\sigma),
\displaystyle \min_{\{i \in p_{z_i}>0\}} \alpha_{z_i}(\sigma) \},
\label{determ}
\end{equation} 
and the following bisection search for variable $\sigma$ to solve (\ref{maxmin}).

\begin{algorithm}  {\bf (bisection search devised for solving (\ref{maxmin})
)} \newline
\upshape 
Data: $(\dot{x},\dot{s})$, $(p_x, p_s )$, $(q_x, q_s)$, $(x^k,s^k)$,
$\phi_k$, and $\psi_k$. {\ } \\
Parameter: $\epsilon \in (0,1)$, $\sigma_{lb}=\sigma_{\min}$,
$\sigma_{ub}=\sigma_{\max} \le 1$. {\ } \\
{\bf for} iteration $k=0,1,2,\ldots$
\begin{itemize}
\item[] Step 0: If $\sigma_{ub}-\sigma_{lb} \le \epsilon$, set
$\alpha=\displaystyle\min_{i \in \{ 1, \ldots,n\} }
\{ \alpha_{x_i}(\sigma), \alpha_{s_i}(\sigma) \}$, stop.
\item[] Step 1: Set $\sigma=\sigma_{lb}+0.5(\sigma_{ub}-\sigma_{lb})$.
\item[] Step 2: Calculate $\alpha_{x_i}(\sigma)$ and 
$\alpha_{s_i}(\sigma)$ using (\ref{case1a})-(\ref{case7b}).
\item[] Step 3: If (\ref{determ}) holds, set $\sigma_{lb}=\sigma$, 
otherwise, set $\sigma_{ub}=\sigma$.
\item[] Step 4: Set $k+1 \rightarrow k$. Go back to Step 1.
\end{itemize}
{\bf end (for)} 
\hfill \qed
\label{bisectionSearch}
\end{algorithm}

This algorithm reduces interval length by $0.5$ in every iteration while 
golden section method reduces interval length by $0.618$. The bisection
is more efficient.

In view of Proposition \ref{simple2}, if 
$(\dot{\s}^{\T} {\p}_{\z}+\dot{\z}^{\T} {\p}_{\s}) < 0$,
to minimize $\mu_{k+1}$, we should select
$\sigma_k =0$. Therefore, Problem (\ref{maxmin}) is reduced to
solve a much simpler problem
\begin{equation}
\tilde{\alpha}= \min_{\alpha} b_u(\alpha).
\label{smaller0}
\end{equation}
This is a one-dimensional unconstrained optimization problem
that can be solved by many existing methods, such as golden
section method.
Given $\tilde{\alpha}$, we still need to find the largest
$\alpha_k \in (0, \tilde{\alpha}]$ such that Condition (C4) holds. 
We summarize the algorithm described above as follows:

\begin{algorithm}  {\bf (bisection search devised for Step 5 of Algorithm \ref{mainAlgo1})
)} \newline
\upshape 
Data: $(\dot{\x},\dot{\s})$, $(\p_{\x}, \p_{\s} )$, $(\q_{\x}, \q_{\s})$, 
$(\x^k,\s^k)$, $\phi_k$, and $\psi_k$. {\ } \\
Parameter: $\epsilon \in (0,1)$. {\ } \\
\indent Step 1: If $(\dot{\s}^{\T} {\p}_{\z}+\dot{\z}^{\T} {\p}_{\s}) < 0$,
set $\sigma_k =0$, solve (\ref{smaller0}) to get $\tilde{\alpha}$. \\
\indent Step 2: Otherwise, call Algorithm \ref{bisectionSearch} to 
get $\tilde{\alpha}$. \\
\indent Step 3: Find the largest $\alpha_k \in (0, \tilde{\alpha}]$
such that Condition (C4) holds.
\hfill \qed
\label{Step5}
\end{algorithm}

\end{document}